\documentclass[letter,11pt]{amsart}
\usepackage[latin1]{inputenc}   
\usepackage{amssymb,amsmath,amsthm,mathrsfs,mathdots}
\usepackage{graphicx,mathpazo}
\usepackage[all]{xy}
\usepackage[usenames,dvipsnames]{color}
\usepackage{enumerate}
\usepackage{stmaryrd}
\usepackage{easyReview}
 
\usepackage{adjustbox}

\usepackage{youngtab,ytableau}
\ytableausetup{centertableaux,boxsize=0.25em}

\usepackage[colorlinks=true,linkcolor=Cerulean,pagebackref,citecolor=Mahogany]{hyperref}

\definecolor{farbe}{RGB}{121,60,0}
\definecolor{Farbe}{RGB}{139,0,139}

 \usepackage{tikz}
\usetikzlibrary{arrows,decorations.pathmorphing,decorations.pathreplacing,positioning,shapes.geometric,shapes.misc,decorations.markings,decorations.fractals,calc,patterns,}
\usepackage{tikz-cd}

\tikzset{>=stealth',
     cvertex/.style={circle,draw=black,inner sep=1pt,outer sep=3pt},
     vertex/.style={circle,fill=black,inner sep=1pt,outer sep=3pt},
     star/.style={circle,fill=yellow,inner sep=0.75pt,outer sep=0.75pt},
     tvertex/.style={inner sep=1pt,font=\criptsize},
     gap/.style={inner sep=0.5pt,fill=white}}

\renewcommand{\AA}{\ensuremath{\mathbb{A}}}

\newcommand{\ZZ}{\ensuremath{\mathbb{Z}}}

\newcommand{\KK}{\ensuremath{K}} 

\DeclareMathOperator{\Sing}{Sing}

\DeclareMathOperator{\Spec}{Spec}

\newcommand{\gen}{{\rm gen}}
\newcommand{\prin}{{\rm prin}}

\newcommand{\cA}{\mathcal{A}}

\newcommand{\cQ}{\mathcal{Q}}

\newcommand{\cX}{\mathcal{X}}

\renewcommand{\c}{{\boldsymbol{c}}} 
  
\newcommand{\q}{\boldsymbol{q}} 
 
\renewcommand{\t}{\boldsymbol{t}} 
\newcommand{\x}{\boldsymbol{x}} 
\newcommand{\y}{\boldsymbol{y}}
\newcommand{\z}{\boldsymbol{z}}

\newcommand{\tx}{\widetilde x}
\newcommand{\ty}{\widetilde y}

\newcommand{\Jac}{\operatorname{Jac}}

\newcommand{\XX}{X}

\newcounter{CountAlpha}

\theoremstyle{theorem}

\newtheorem{MainThm}[CountAlpha]{Theorem}  

\newtheorem{Thm}{Theorem}[section]         
\newtheorem{lemma}[Thm]{Lemma}
\newtheorem{cor}[Thm]{Corollary}

\newtheorem{prop}[Thm]{Proposition}

\newtheorem{Qu}[Thm]{Question}

\theoremstyle{definition}
\newtheorem{defi}[Thm]{Definition} 

\newtheorem{example}[Thm]{Example}
\newtheorem{Bem}[Thm]{Remark}

\newcommand{\len}{\ell}
\newcommand{\halflen}{k}

\numberwithin{equation}{section}

\title{Singularities of star cluster algebras}

\author{Eleonore Faber}
\address{
	Institut f\"ur Mathematik und Wissenschaftliches Rechnen,
	Universit\"at Graz,
	Heinrichstr.~36,
	A-8010 Graz, Austria and School of Mathematics, University of Leeds, LS2 9JT Leeds, UK
}
\email{e.m.faber@leeds.ac.uk, eleonore.faber@uni-graz.at}

\author{Bernd Schober}
\address{
	None. (Hamburg, Germany).
}
\email{schober.math@gmail.com}

\date{\today}

\makeatletter
\@namedef{subjclassname@2020}{\textup{2020} Mathematics Subject Classification}
\makeatother



\begin{document}

\maketitle

\begin{abstract}
	We introduce the class of star cluster algebras and classify their singularities. 
	Then we focus on the combinatorial structure of the desingularization by determining the number of irreducible centers that are blown up.
\end{abstract}


\section{Introduction}

In recent joint work with Benito and Mourtada \cite{BFMS23,BFMS25} we investigated the singularities of cluster algebras of finite cluster type. 
Additionally, \cite[Sec.~6]{BFMS23} classified the singularities of cluster algebras with trivial coefficients arising from a star-shaped quiver as shown in Fig.~\ref{Fig:Ex_star_quiver_BFMS23}.
A combinatorial analysis of the irreducible components of the singular locus showed that their 
number is equal to $ m (m-1) 2^{m-3} $, where $ m $ is the valency of the vertex in the middle, see \cite[Thm.~6.3 and Remark~6.5]{BFMS23}. 

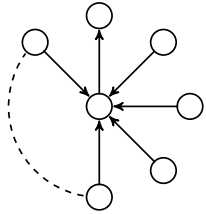
\begin{figure}[h!] 
	\scalebox{0.8}{
	\begin{tikzpicture}[->,>=stealth',shorten >=1pt,auto,node distance=1.5cm, thick,main node/.style={circle,draw,white}]

		\node[main node] (0) { };
		\node[main node] (1) [above right of=0] { };
		\node[main node] (2) [right of=0] { };
		\node[main node] (3) [below right of=0] { };
		\node[main node] (4) [below of=0] { };
		\node[main node] (m-1) [above left of=0] { };
		\node[main node] (m) [above  of=0] { };


		\draw (0) circle (6pt);
		\draw (1) circle (6pt);
		\draw (2) circle (6pt);
		\draw (3) circle (6pt);
		\draw (4) circle (6pt);
		\draw (m-1) circle (6pt);
		\draw (m) circle (6pt);

		\path

		(3) edge (0)
		(1) edge (0)
		(2) edge (0)
		(4) edge (0)
		(m-1) edge (0)
		(0) edge (m);

		\draw[dashed,-] (260:1.5cm) arc (260:143:1.5cm);

	\end{tikzpicture}
}
	\caption{First example of a star-shaped quiver as appeared in \cite{BFMS23}.}
	\label{Fig:Ex_star_quiver_BFMS23}
\end{figure}

There also exists prior work on the singularity type of cluster algebras. 
In \cite{BMRS2015} it was shown that 
locally acyclic cluster algebras have at worst canonical singularities in characteristic zero and that they are
strongly $ F $-regular in positive characteristic.
Furthermore, \cite{MRZ2018} proved Cohen--Macaulayness and normality for the lower bound cluster algebra, which coincides with the corresponding cluster algebra for acyclic quivers. 
But no combinatorial analysis of the singularities has been made apart from \cite{BFMS23}.

In the present article we initiate a more thorough study of star cluster algebras, their singularity theory and the corresponding combinatorics.
We introduce a generalization of cluster algebras arising from star-shaped quivers where the rays of the star are of possibly different lengths (Def.~\ref{Def:star}),
as illustrated in Fig.~\ref{Fig:Ex_star_quiver}.

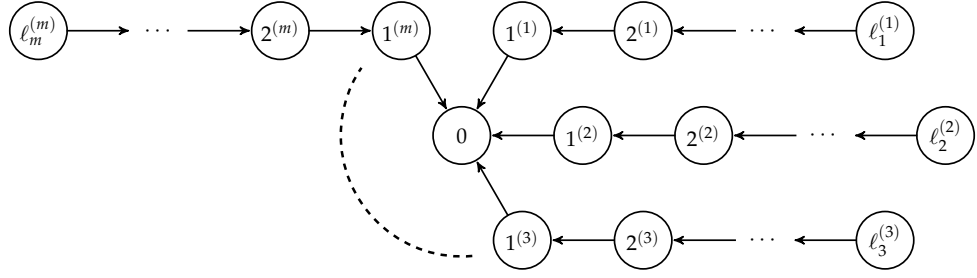
\begin{figure}[h!] 
	\scalebox{0.8}{
	\begin{tikzpicture}[->,>=stealth',shorten >=1pt,auto,node distance=2cm, thick,main node/.style={circle,draw,white}]

		\node[main node] (0) {1 1};
		
		\node[main node] (11) at (60:2cm) {1 1};
		\node[main node] (21) [right of=11] {1 1};
		\node[main node] (x1) [right of=21] {1 1};
		\node[main node] (l1) [right of=x1] {1 1};

		\node[main node] (12) at (0:2cm) {1 1};
		\node[main node] (22) [right of=12] {1 1};
		\node[main node] (x2) [right of=22] {1 1};
		\node[main node] (l2) [right of=x2] {1 1};

		\node[main node] (13) at (-60:2cm) {1 1};
		\node[main node] (23) [right of=13] {1 1};
		\node[main node] (x3) [right of=23] {1 1};
		\node[main node] (l3) [right of=x3] {1 1};
		
		\node[main node] (1m) at (120:2cm) {1 1};
		\node[main node] (2m) [left of=1m] {1 1};
		\node[main node] (xm) [left of=2m] {1 1};
		\node[main node] (lm) [left of=xm] {1 1};

		\node at (0) {\small $0$};
		
		\node at (11) {\small $ 1^{(1)} $};
		\node at (21) {\small $ 2^{(1)} $};
		\node at (x1) {\small $ \cdots $};
		\node at (l1) {\small $ \len_1^{(1)} $};	
		
		\node at (12) {\small $ 1^{(2)} $};
		\node at (22) {\small $ 2^{(2)} $};
		\node at (x2) {\small $ \cdots $};
		\node at (l2) {\small $ \len_2^{(2)} $};
		
		\node at (13) {\small $ 1^{(3)} $};
		\node at (23) {\small $ 2^{(3)} $};
		\node at (x3) {\small $ \cdots $};
		\node at (l3) {\small $ \len_3^{(3)} $};
		
		\node at (1m) {\small $ 1^{(m)} $};
		\node at (2m) {\small $ 2^{(m)} $};
		\node at (xm) {\small $ \cdots $};
		\node at (lm) {\small $ \len_m^{(m)} $};

		\draw (0) circle (13.5pt);
		
		\draw (11) circle (13.5pt);
		\draw (21) circle (13.5pt);
		\draw (l1) circle (13.5pt);
		
		\draw (12) circle (13.5pt);
		\draw (22) circle (13.5pt);
		\draw (l2) circle (13.5pt);
		
		\draw (13) circle (13.5pt);
		\draw (23) circle (13.5pt);
		\draw (l3) circle (13.5pt);
		
		\draw (1m) circle (13.5pt);
		\draw (2m) circle (13.5pt);
		\draw (lm) circle (13.5pt);

		\path
	
		(11) edge (0)
		(21) edge (11)
		(x1) edge (21)
		(l1) edge (x1)
		(12) edge (0)
		(22) edge (12)
		(x2) edge (22)
		(l2) edge (x2)
		(13) edge (0)
		(23) edge (13)
		(x3) edge (23)
		(l3) edge (x3)
		(1m) edge (0)
		(2m) edge (1m)
		(xm) edge (2m)
		(lm) edge (xm);

		 \draw[very thick, dashed,-] (275:2cm) arc (275:145:2cm);
		
	\end{tikzpicture}
}
	\caption{Example of a star-shaped quiver with rays of different lengths.}
	\label{Fig:Ex_star_quiver}
\end{figure}


Star-shaped quivers have shown up in a representation theoretic context before.
For example, they appear as quivers of Ringel's canonical algebras \cite{Ringel84}
as well as in connection with the Deligne-Simpson problem \cite{C-B_S_2006, C-B_H_2025}.
Moreover, in \cite{IW_weighted}, Iyama and Wemyss studied some rational surfaces with star-shaped dual resolution graphs via a Geigle--Lenzing weighted projective line \cite{GeigleLenzing87}.

Observe that star-shaped quivers cover all quivers of type $ A, D, E $ (see~Fig.~\ref{Fig:ADE})
and more generally trees of type $ T_{p,q,r} $.
Therefore, Thm.~\ref{Thm:A} below provides a unified presentation of the classification of the singularities of cluster algebras of finite cluster type $ A, D, E $ 
(cf.~\cite[Thm.~A,~D,~E]{BFMS25}, and also \cite[Thm.~A]{BFMS23}).
\\
Since star shaped quivers with less than three rays are quivers of type $ A_n $ for some $ n $,
the singularities of the corresponding cluster algebra 
have been classified in \cite[Thm.~A]{BFMS25}, see also \cite[Thm.~A]{BFMS23}.
Therefore, we may assume that the valency of the vertex in the middle is at least three.

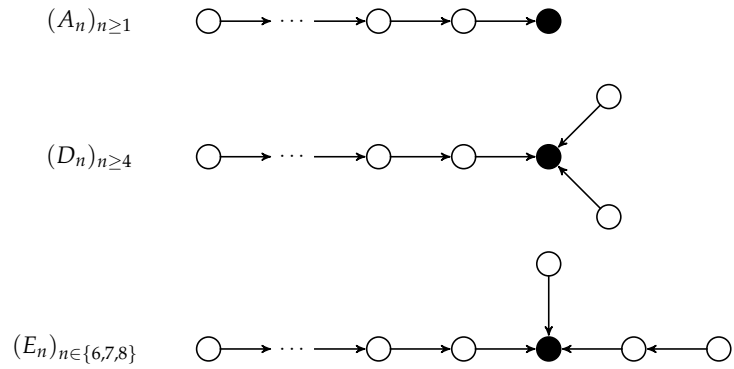
\begin{figure}[h!] 
	\scalebox{0.75}{
		\begin{tikzpicture}[->,>=stealth',shorten >=1pt,auto,node distance=1.5cm, thick,main node/.style={circle,draw,white}]


			\node[main node] (0) { };
			\node[main node] (1) [left of=0] { };
			\node[main node] (2) [left of=1] { };
			\node[main node] (x) [left of=2] {1};
			\node[main node] (m) [left of=x] { };
			\node[main node] (An) [left of=m] { };
			
			\node at (x) {\small $ \cdots $};
			\node at (An) {\large $ (A_n)_{n\geq 1} $ \hspace{1cm} \ };

			\draw[fill=black] (0) circle (6pt);
			\draw (1) circle (6pt);
			\draw (2) circle (6pt);
			\draw (m) circle (6pt);

			\path
			
			(1) edge (0)
			(2) edge (1)
			(x) edge (2)
			(m) edge (x);


			\node[main node] (0d) [below = 2cm of 0] { };
			\node[main node] (1d) [left of=0d] { };
			\node[main node] (2d) [left of=1d] { };
			\node[main node] (xd) [left of=2d] {1};
			\node[main node] (md) [left of=xd] { };
			\node[main node] (Dn) [left of=md] { };
			\node[main node] (ad) [above right of=0d] { };
			\node[main node] (bd) [below right of=0d] { };

			\node at (xd) {\small $ \cdots $};
			\node at (Dn) {\large $ (D_n)_{n\geq 4} $ \hspace{1cm} \ };

			\draw[fill=black] (0d) circle (6pt);
			\draw (1d) circle (6pt);
			\draw (2d) circle (6pt);
			\draw (md) circle (6pt);
			\draw (ad) circle (6pt);
			\draw (bd) circle (6pt);

			\path
			(ad) edge (0d)
			(bd) edge (0d)
			(1d) edge (0d)
			(2d) edge (1d)
			(xd) edge (2d)
			(md) edge (xd);


			\node[main node] (0e) [below = 3cm of 0d] { };
			\node[main node] (1e) [left of=0e] { };
			\node[main node] (2e) [left of=1e] { };
			\node[main node] (xe) [left of=2e] {1};
			\node[main node] (me) [left of=xe] { };
			\node[main node] (En) [left of=me] { };
			\node[main node] (Oe) [above of=0e] { };
			\node[main node] (Ae) [right of=0e] { };
			\node[main node] (Be) [right of=Ae] { };

			\node at (xe) {\small $ \cdots $};
			\node at (En) {\large $ (E_n)_{n \in \{6,7,8\}} $ \hspace{1.5cm} \ };

			\draw[fill=black] (0e) circle (6pt);
			\draw (1e) circle (6pt);
			\draw (2e) circle (6pt);
			\draw (me) circle (6pt);
			\draw (Oe) circle (6pt);
			\draw (Ae) circle (6pt);
			\draw (Be) circle (6pt);

			\path
			(Oe) edge (0e)
			(Ae) edge (0e)
			(Be) edge (Ae)
			(1e) edge (0e)
			(2e) edge (1e)
			(xe) edge (2e)
			(me) edge (xe);

		\end{tikzpicture}
	}
	\caption{Quivers corresponding to cluster algebras of finite cluster type with the central vertex marked black}
	\label{Fig:ADE}
\end{figure}

As in \cite{BFMS25}, we follow the viewpoint of 
\cite{BFMN_degen} 
and encode coefficients of a cluster algebra
by considering the corresponding varieties as families and then studying the respective fibers,
see the beginning of Sec.~\ref{Sec:Classi}.
In particular, we consider cluster algebras with generic coefficients (Def.~\ref{Def:gen_coeff}) and use the reduction of \cite{BFMS25} to principal coefficients (Lemma~\ref{Lem:prin}) which simplifies the investigation of the singularities.
Therefore, 
we are lead to a family 
\[ 
	\phi \colon \Spec(\cA_\star^\prin) \to  S := \Spec(\KK [\c^\pm] ) 
\]
where $ \c $ are the coefficients and $ \KK $ is the base field.
Note that we will deduce a presentation of $ \Spec(\cA_\star^\prin) $ embedding it into $ \AA_\kappa^{\len_1+\cdots + \len_m+m+2} $
(see Prop.~\ref{Prop:NewPresentation}). 
The main object of our investigation are the fibers
\[  
	X := \phi^{-1} (\eta) 
	\subseteq \AA_\kappa^{\len_1+\cdots + \len_m+m+2} 
\ ,
\] 
where $ \eta \in S $ is a closed point and $ \kappa = \kappa (\eta ) $ denotes the residue field of the local ring $ \mathcal{O}_{S,\eta} $.
In Prop.~\ref{Prop:NewPresentation}, we deduce explicit defining equations for $ X $ which make the problem to determine and classify the singularities feasible.

Recall that a hypersurface singularity of dimension $ n $ is called {\em of type $ A_1 $} if it has an isolated singularity at which the variety is locally given by 
$  x_0^2 + x_1 x_2 + \cdots + x_{n-1} x_n  = 0 $ (if $ n $ is even)
resp.~by $ x_0 x_1 + \cdots + x_{n-1} x_n  = 0 $ (if $ n $ is odd)
after possibly passing to the completion at the singular point. 
In particular, the resolution of such a singularity is obtained by simply blowing up the singular locus. 
\\
We introduce the notion of an intersection of compatible $ A_1 $-hypersurface singularities
and provide an explicit description of its desingularization, 
see Lemma~\ref{Lem:intersec_of_An} and the paragraph after its statement.
\\
If a singularity $ Z \subset \AA_\kappa^N $ is locally isomorphic to $ Y \times \AA_\kappa^{N-M} $ for a singularity $ Y \subset \AA_\kappa^M $,
then we say that $ Z $ is a {\em cylinder over $ Y $}. 

Our main result on the classification of the singularities of $ X $ formulates as follows:

\begin{MainThm}[Thm.~\ref{Thm:Sing} and~\ref{Thm:ClassSing}]
	\label{Thm:A}
	Let $ X := \phi^{-1} (\eta) 
	\subseteq \AA_\kappa^{\len_1+\cdots + \len_m+m+2} $ 
	be as defined above. 
	The singularities of $ X $ have the following classification:
	\begin{enumerate}[(1)]
		\item 
		Either $ X $ has an isolated singularity locally at which $ X $ is isomorphic to a hypersurface singularity of type $ A_1 $, 
		
		\item 
		or $ Sing(X) \cong  \bigcup_{i=1}^s  C_i  $ has $ s $ disjoint components (where $ C_i $ is not necessarily irreducible)
		and for each $ i \in \{ 1, \ldots, s \} $,
		$ X $ is locally at $ C_{i} $ isomorphic to a cylinder over an intersection of compatible $ A_1 $-hypersurface singularities, 
		
		\item 
		or $ X $ is regular.
	\end{enumerate}
\end{MainThm}

In Thm.~\ref{Thm:Sing} and~\ref{Thm:ClassSing} we describe explicitly the conditions under which the cases of the distinction outlined in Thm.~\ref{Thm:A} appear.
The precise conditions are based on technical notation introduced in Prop.~\ref{Prop:NewPresentation} and Def.~\ref{Def:partition}.


As in \cite{BFMS23}, we are interested in combinatorial aspects of the singularities. 
The explicit description of the desingularization of an intersection of compatible hypersurfaces singularities of type $ A_1 $ provided in Lemma~\ref{Lem:intersec_of_An}
allows to count the number of irreducible components of a center in each step of the resolution, see Prop.~\ref{Prop:cA}.
This leads to the following integer sequence for the total number of centers that are blown up.

\begin{MainThm}[Prop.~\ref{Prop:cA}]
	\label{Thm:B}
	Let $ \cX^{(n)} $ be an intersection of $ n $ compatible $ A_1 $-hypersurface singularities
	where the $ A_1 $-hypersurface singularities are contained in $ \AA_\kappa^{2 q_a} $ for $  a \in \{ 1, \ldots, n \} $
	for $ \kappa $ a field. 
	Set $ \q := (q_1, \ldots, q_n) $
	and $ \alpha := \# \{ a \in \{ 1, \ldots, n \} \mid q_a = 1 \} $. 
	Let $ \mathcal{A} (n;\q) $ be the number of irreducible components blown up along the desingularization of 
	$  \cX^{(n)} $. 
	Then, we have: 
	\[
	\mathcal{A} (n;\q) 
	\ = \ 
	3^\alpha 2^{n-\alpha}  
	- ( n + 1) 2^{\alpha} 
	+ \alpha 2^{\alpha - 1}
	\ .
	\]
\end{MainThm}

The presented statement of Thm.~\ref{Thm:B} is a simplification of Prop.~\ref{Prop:cA} which provides more refined information also addressing the number of irreducible components blown up in each step of the desingularization procedure of Lemma~\ref{Lem:intersec_of_An}.

Let us point out that in boundary cases the formula for $ \cA(n;\q) $ leads to Eulerian numbers ($ \alpha = 0 $) resp.~to the binomial transform of the latter ($ \alpha = n $),
see Cor.~\ref{cor:bondary}.



{\em Summary.}
In Sec.~\ref{Sec:Star} we introduce the notion of star cluster algebras and recall the required background on cluster algebras. 
In particular, we deduce a presentation of the commutative algebra of a star cluster algebra using continuant polynomials 
which will be useful for our task to study the singularities. 
After that, we classify the corresponding singularities in Sec.~\ref{Sec:Classi}, proving Thm.~\ref{Thm:A}.
This includes the definition and investigation of an intersection of compatible $ A_1 $-hypersurfaces. 
In Sec.~\ref{Sec:Combi}, we turn our attention to combinatorial aspects of the singularities.
In particular, we show Thm.~\ref{Thm:B} and draw the connection to Eulerian numbers.
After that, we end by discussing possible research direction to proceed further with in Sec.~\ref{Sec:beyond}.

Throughout the article
$ \KK $ will be a field of any characteristic.
Furthermore, we follow the convention that coefficient variables are always assumed to be invertible.

{\em Acknowledgements:}
We thank the Mathematisches Forschungs\-institut Oberwolfach for the excellent research environment and ideal working conditions during our stays as Oberwolfach Research Fellows.  
Further, we thank Ang\'elica Benito and Hussein Mourtada for initial discussions motivating us to continue in this more general direction. 
This work was supported by EPSRC grant EP/W007509/1 and by NAWI Graz

\section{Star Cluster Algebras}
\label{Sec:Star}

First, we introduce the notion of a star cluster algebra and deduce a presentation of the corresponding commutative algebra which is suitable to investigate its singularity theory. 

Recall that a {\em quiver} $ \cQ = (\cQ_0, \cQ_1) $ is a finite directed graph consisting of the finite set of vertices $ \cQ_0 = \{ 1, \ldots, N \} $ and a finite multi-set $ \cQ_1 $ of arrows between the vertices.
The elements of $ \cQ_1 $ are written as pairs $ (i,j) $ corresponding to arrows $ i \to j $.
Since multiple arrows between two vertices are possible, $ \cQ_1 $ has to be a multi-set. 
In the context of cluster algebras, the following two hypotheses are imposed on $ \cQ $: 
there are no loops in $ \cQ_1 $ (i.e., for every $ i \in \cQ_0 $, we have $ (i,i) \notin \cQ_1 $)
and there is no oriented $ 2 $-cycle in $ \cQ $ (i.e., $ (i,j) \in \cQ_1 $ implies $ (j,i) \notin \cQ_1 $).

A quiver is called {\em acyclic} if it does not contain an oriented cycle.
The quivers illustrated in Fig.~\ref{Fig:Ex_star_quiver_BFMS23}, Fig.~\ref{Fig:Ex_star_quiver}, and Fig.~\ref{Fig:ADE} are examples of acyclic quivers. 

An essential construction for the definition of cluster algebras is the notion of mutation. 
In terms of quivers,
the mutation of $ \cQ $ at a vertex $ k \in \cQ_0 $ is a new quiver $ \cQ' = \mu_k (\cQ) $
which is constructed from $ \cQ $ by
first adding for each directed path $ i \to k \to j $ in $ \cQ $ a new arrow $ i \to j $,
then reversing all arrows incident to $ k $,
and finally removing all oriented $ 2 $-cycles from the resulting quiver.
Two quivers are called {\em mutation-equivalent} if there exists a finite sequence of mutations transforming one quiver into the other. 

Using mutation, the set of vertices can be refined by distinguishing it into 
{\em mutable} vertices in whose direction it is allowed to perform a mutation 
and  {\em frozen} vertices for which it is not allowed to perform a mutation. 
Mutable vertices are usually marked with circles and frozen vertices by squares when drawing quivers.
A quiver is called {\em totally mutable} if all its vertices are mutable.
In particular, all preceding figures show only totally mutable quivers.
In Fig.~\ref{Fig:Ex_frozen}, we provide an example of a quiver with frozen vertices. 

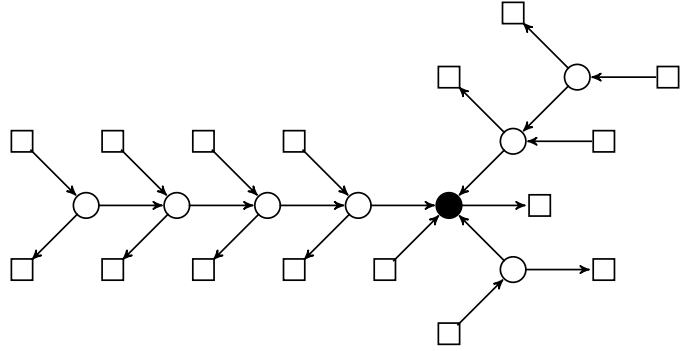
\begin{figure}[h!] 
	\scalebox{0.8}{
	\begin{tikzpicture}[->,>=stealth',shorten >=1pt,auto,node distance=1.5cm, thick,main node/.style={circle,draw,white}]

		\node[main node] (0) { };
		\node[main node] (0a) [below left of=0] { };
		\node[main node] (0b) [right of=0] { };
		\node[main node] (1) [left of=0] { };
		\node[main node] (1a) [above left of=1] { };
		\node[main node] (1b) [below left of=1] { };
		\node[main node] (2) [left of=1] { };
		\node[main node] (2a) [above left of=2] { };
		\node[main node] (2b) [below left of=2] { };
		\node[main node] (3) [left of=2] { };
		\node[main node] (3a) [above left of=3] { };
		\node[main node] (3b) [below left of=3] { };
		\node[main node] (4) [left of=3] { };
		\node[main node] (4a) [above left of=4] { };
		\node[main node] (4b) [below left of=4] { };
		\node[main node] (A) [above right of=0] { };
		\node[main node] (Aa) [right of=A] { };
		\node[main node] (Ab) [above left of=A] { };
		\node[main node] (B) [above right of=A] { };
		\node[main node] (Ba) [right of=B] { };
		\node[main node] (Bb) [above left of=B] { };
		\node[main node] (a) [below right of=0] { };
		\node[main node] (aa) [below left of=a] { };
		\node[main node] (ab) [right of=a] { };

		\draw[fill=black] (0) circle (6pt);
		\draw ([xshift=-5pt,yshift=-5pt]0a) rectangle ++(10pt,10pt);
		\draw ([xshift=-5pt,yshift=-5pt]0b) rectangle ++(10pt,10pt);
		\draw (1) circle (6pt);
		\draw ([xshift=-5pt,yshift=-5pt]1a) rectangle ++(10pt,10pt);
		\draw ([xshift=-5pt,yshift=-5pt]1b) rectangle ++(10pt,10pt);
		\draw (2) circle (6pt);		
		\draw ([xshift=-5pt,yshift=-5pt]2a) rectangle ++(10pt,10pt);
		\draw ([xshift=-5pt,yshift=-5pt]2b) rectangle ++(10pt,10pt);
		\draw (3) circle (6pt);		
		\draw ([xshift=-5pt,yshift=-5pt]3a) rectangle ++(10pt,10pt);
		\draw ([xshift=-5pt,yshift=-5pt]3b) rectangle ++(10pt,10pt);
		\draw (4) circle (6pt);		
		\draw ([xshift=-5pt,yshift=-5pt]4a) rectangle ++(10pt,10pt);
		\draw ([xshift=-5pt,yshift=-5pt]4b) rectangle ++(10pt,10pt);
		\draw (A) circle (6pt);
		\draw ([xshift=-5pt,yshift=-5pt]Aa) rectangle ++(10pt,10pt);
		\draw ([xshift=-5pt,yshift=-5pt]Ab) rectangle ++(10pt,10pt);
		\draw (B) circle (6pt);
		\draw ([xshift=-5pt,yshift=-5pt]Ba) rectangle ++(10pt,10pt);
		\draw ([xshift=-5pt,yshift=-5pt]Bb) rectangle ++(10pt,10pt);
		\draw (a) circle (6pt);
		\draw ([xshift=-5pt,yshift=-5pt]aa) rectangle ++(10pt,10pt);
		\draw ([xshift=-5pt,yshift=-5pt]ab) rectangle ++(10pt,10pt);

		\path
		(0a) edge (0)
		(0) edge (0b)
		(A) edge (0)
		(Aa) edge (A)
		(A) edge (Ab)
		(B) edge (A)
		(Ba) edge (B)
		(B) edge (Bb)
		(a) edge (0)
		(aa) edge (a)
		(a) edge (ab)
		(1) edge (0)
		(1a) edge (1)
		(1) edge (1b)
		(2) edge (1)
		(2a) edge (2)
		(2) edge (2b)
		(3) edge (2)
		(3a) edge (3)
		(3) edge (3b)
		(4) edge (3)
		(4a) edge (4)
		(4) edge (4b);

	\end{tikzpicture}
}
	\caption{Example of a quiver with vertices that are frozen (marked as squares) and mutable (marked as circles), and a central mutable vertex marked in black.}
	\label{Fig:Ex_frozen}
\end{figure}

A totally mutable quiver $ \cQ $ is {\em of type $ A_N $} if it is mutation equivalent to some $ A_N $-quiver,
i.e., one of the form in Fig.~\ref{Fig:ADE}.

\begin{defi}
	\label{Def:star}
	Let $ \cQ = (\cQ_0,\cQ_1) $ be a totally mutable quiver. 
	We say that $ \cQ $ is of {\em type star-shaped} if it is mutation equivalent to a star-shaped quiver $ \cQ_\star $,
	i.e., a quiver that   
	is made up from the following components:
	\begin{enumerate}[(1)]
		\item 
		finitely many quivers $ \cQ^{(\alpha)}_\star $ of type $ A_{\len_\alpha} $ with 
		respective set of vertices $ \{ 1^{(\alpha)}, \ldots, \len_\alpha^{(\alpha)} \} $ 
		and $ \len_\alpha \in \ZZ_+ $, 
		for $ \alpha \in \{ 1, \ldots, m \} $,
		and
		
		\item 
		an additional, distinguished vertex, say $ 0 $, 
		and there exists a single arrow between $ 0 $ and $ 1^{(\alpha)}  $, 
		for all $ \alpha \in \{ 1, \ldots, m \} $. 
	\end{enumerate} 
	We call 
	$ \cQ^{(\alpha)}_\star  $ the {\em rays} of $ \cQ_\star $,
	$ \len_\alpha $ the {\em length of the ray} $ \cQ^{(\alpha)}_\star $,
	and
	$ m $ the {\em number of rays} of $ \cQ_\star $.  
\end{defi}

Note that $ \cQ_\star $ has $ \len_1 + \cdots +  \len_m + 1 $ vertices. 
Clearly, a star-shaped quiver is acyclic
since its underlying graph 
(i.e., the graph that we obtain after forgetting the orientation of the arrows) 
is a tree. 
The latter also implies that two star-shaped quivers with possibly different orientation on the arrows are mutation-equivalent,
see \cite[Lemma~5.1.4]{Marsh13}. 
Therefore, we may assume without loss of generality 
that the orientation of the arrows of a star shaped quiver are as given in Fig.~\ref{Fig:Ex_star_quiver}.

Let us next explain how a cluster algebra $ \cA(\cQ) $ is associated to an acyclic quiver $ \cQ $,
which is sufficient for the considerations in the present article. 
Instead of recalling the definition in details, we introduce the cluster algebra $ \cA(\cQ) $ 
via the following theorem, which is a consequence of \cite[Thm.~1.20 and Cor.~1.21]{BFZ2005} 
and which gives a suitable presentation of $ \cA(\cQ) $.
For the general definition of a cluster algebra we refer to \cite[Section~2]{FZ2002}. 

\begin{Thm}[{\cite{BFZ2005}}]
	\label{Thm:FirstPresentation}
	Let $ \cQ $ be an acyclic quiver with $ N $ vertices such that the vertices $ 1, \ldots, n $ are mutable 
	and $ n+1, \ldots, N $ are frozen.  
	Let $ K $ be a field
	and let $ \x = (x_1, \ldots, x_N ) $ be an $ N $-tuple of elements that are algebraically independent over $ K $.
	The cluster algebra $ \cA(\cQ) \subset K(\x) =  K(x_1, \ldots, x_n, \ldots, x_N) $ (defined over $ K $) corresponding to $ \cQ $ is isomorphic to
	\begin{equation}
		\label{eq:presentation1} 
		K[x_{n+1}^{\pm 1}, \ldots, x_{N}^{\pm 1}][x_1, \ldots, x_n, x_1', \ldots, x_n']
		/ \langle x_k x_k' - \prod_{i\rightarrow k} x_i^{b_{ik}} - \prod_{i \leftarrow k} x_i^{-b_{ik}} \mid k \in \{ 1, \ldots, n \} \rangle 
	\end{equation}
	where $ (x_1', \ldots, x_n') $ is another tuple of algebraically independent elements over $ K $, 
	the products range over those vertices $ i $ of $ \cQ $ for which there is an arrow $ i \rightarrow k $ (resp.~$ i \leftarrow k $),
	and
	$ b_{ik} \in \ZZ $ is the number of arrows $ i \rightarrow k $ which is negative if there are arrows $ i \leftarrow k $. 
	\\
	Moreover, 
	$ (x_k x_k' - \prod_{i\rightarrow k} x_i^{b_{ik}} - \prod_{i \leftarrow k} x_i^{-b_{ik}} \mid k \in \{ 1, \ldots, n \}) $ 
	is a Gr\"obner basis for the ideal which it generates with respect to any term order
	giving preference to $ x_1', \ldots, x_n' $ over $ \x $.  
\end{Thm}

Recall that there are no loops and no oriented $ 2 $-cycles in a quiver.
Therefore, the integers $ b_{ik} $ in \eqref{eq:presentation1} are well-defined
and the matrix $ (b_{ik})_{i,k \in \{ 1, \ldots, n \} } $ is skew-symmetric.

The variables corresponding to frozen vertices of a quiver are sometimes also called {\em coefficients}.
For a more detailed presentation of coefficients through tropical semi-fields, we refer to \cite[Section~2]{FZ2002}, \cite{FZ2007}, or \cite[Def.~3.1.1]{FWZ2016}.
If $ n = N $ in Thm.~\ref{Thm:FirstPresentation}, 
then all vertices are mutable and the cluster algebra is said to have {\em trivial coefficients}. 

We now introduce the central notion of our present work
generalizing \cite[Sec.~6]{BFMS23}. 

\begin{defi}
	A cluster algebra is called {\em star cluster algebra}
	if it is equal to $ \cA(\cQ) $ 
	for a quiver $ \cQ  $ of type star-shaped.
\end{defi}

\underline{\em Convention:}
Throughout the article we focus on star cluster algebras of totally mutable quivers $ \cQ $. 
Nonetheless we consider coefficients by adding additional frozen vertices and arrows to $ \cQ $ introducing suitable coefficients in the defining equations for $ \cA (\cQ) $ discussed in Thm.~\ref{Thm:FirstPresentation}, 
as we will explain in the following.

In order to understand the singularities of the commutative algebra determined by a cluster algebra with coefficients 
\cite{BFMS25} introduced the notion of a cluster algebra with generic coefficients.
Using them a unified approach to the respective singularity theory is possible. 
Let us recall \cite[Def.~3.2.1]{BFMS25} in the language of quivers, cf.~\cite[Rk.~3.2.2.(4)]{BFMS25}.

\begin{defi}
	\label{Def:gen_coeff}
	Let $ \cQ $ be a quiver.
	We define a new quiver $ \cQ^\gen $ in the following way: 
	For every mutable vertex $ k $ of $ \cQ $, 
	add two new frozen vertices $ k_s $ and $ k_t $ as well as exactly one arrow $ k_s \rightarrow k $ and another one $ k \to k_t $. 
	We call the corresponding cluster algebra $ \cA(\cQ^\gen) $ the 
	{\em cluster algebra with generic coefficients} associated to $ \cQ $. 
\end{defi}

The quiver pictured in Fig.~\ref{Fig:Ex_frozen} is $ \cQ_\star^\gen $ for a star-shaped quiver $ \cQ_\star $ with 
three rays of lengths one, two and four, respectively.

As we assume coefficients to be invertible,
it is possible to connect the cluster algebra with generic coefficients to the {\em cluster algebra with principal coefficients}. 
The latter is the cluster algebra $ \cA(\cQ^\prin) $
where $ \cQ^\prin $ is constructed from $ \cQ $ by adding for every mutable vertex $ k $ in $ \cQ $ one new frozen vertex $ k_c $ and one arrow $ k_c \to k $. 
For more details on this notion, we refer to \cite[Def.~3.1 and Rk.~3.2]{FZ2007}
or \cite[Sec.~4.1]{INT_complexes}.

\begin{lemma}[{\cite[Lemma~3.3.2]{BFMS25}}]
	\label{Lem:prin}
	Let $ \cQ $ be an acyclic quiver with $ N $ vertices of which $ 1, \ldots, n $ are mutable for some $ n \leq N $ and let $ K $ be a field. 
	There is the following isomorphism for the 
	spectra of the
	corresponding cluster algebras with generic and principal coefficients defined over $ K $, 
	\[
		\Spec(\cA ( \cQ^\gen ) ) \cong \Spec (\cA ( \cQ^\prin ) ) \times (K^\times)^n  
		\ .
	\]
\end{lemma}

This result motivates to focus on the singularity theory of cluster algebras with principal coefficients.

We end this section by deducing a suitable presentation of the cluster algebra with principal coefficients associated to a star-shaped quiver:
Let $ \cQ_\star $ be a star-shaped quiver and $ m $ rays which are of length $ \len_\alpha $ for $ \alpha \in \{ 1, \ldots, m \} $. 
Let $ \KK $ be a field.
Let  
\[  
	\c = (c_0, c_1^{(1)}, \ldots, c_{\len_1}^{(1)}, \ldots, c_1^{(\alpha)}, \ldots, c_{\len_\alpha}^{(\alpha)}, \ldots, c_1^{(m)}, \ldots, c_{\len_m}^{(m)} ) 
\] 
be algebraically independent elements
corresponding to the frozen vertices in $ \cQ_\star^\prin $.
We use the abbreviations
\[
	\cA_\star^\prin :=\cA(\cQ_\star^\prin )
	\quad 
	\mbox{ and } 
	\quad 
	K_\c := K[\c^{\pm 1}]
	= K\left[c_0^{\pm 1}, \ldots, c_{\len_m}^{(m) \pm 1}  \right]
	\ . 
\]

Let $ (\x, u ) = (x_1^{(1)}, \ldots, x_{\len_1}^{(1)}, \ldots, x_1^{(\alpha)}, \ldots, x_{\len_\alpha}^{(\alpha)}, \ldots, x_1^{(m)}, \ldots, x_{\len_m}^{(m)}, u ) $ be algebraically independent elements
corresponding to the mutable vertices of $ \cQ_\star^\prin $. 
By Thm.~\ref{Thm:FirstPresentation},
we have the following presentation of $ \cA_\star^\prin $:  
\begin{equation} 
	\label{eq:A_star}
	\cA_\star^\prin \cong 
	\KK_\c [\x , \y,  u, v  ] / \langle g,  f_i^{(\alpha)} \mid \alpha \in \{ 1, \ldots, m \},
\ 
i \in \{1, \ldots, \len_\alpha \} \rangle \ ,  
\end{equation} 
where $ (\y, v ) = (y_1^{(1)}, \ldots, y_{\len_1}^{(1)}, \ldots, y_1^{(\alpha)}, \ldots, y_{\len_\alpha}^{(\alpha)}, \ldots, y_1^{(m)}, \ldots, y_{\len_m}^{(m)}, v ) $ are additional algebraically independent elements
and
\begin{equation}
	\label{eq:gen-original}
	\left\{ \quad
	\begin{array}{rcll}
		g 
		& := &	
		\displaystyle 
		u v  - c_0 \prod_{\alpha=1}^m x_1^{(\alpha)} - 1 \ ,
		\\[15pt]
		%
		%
		f_{i}^{(\alpha)} 
		& := &
		x_i^{(\alpha)} y_i^{(\alpha)} - x_{i-1}^{(\alpha)} - c_i^{(\alpha)} x_{i+1}^{(\alpha)}  \ ,
		&
		\ \ \ 
		\mbox{for } i \in \{ 1, \ldots, \len_\alpha - 1 \} \ ,
		\\[10pt]
		f_{\len_\alpha}^{(\alpha)} 
		& := &
		x_{\len_\alpha}^{(\alpha)} y_{\len_\alpha}^{(\alpha)} - x_{\len_\alpha-1}^{(\alpha)} - c_{\len_\alpha}^{(\alpha)}   \ , 
	\end{array}
	\right.
\end{equation}
for $ \alpha \in \{ 1, \ldots, m \} $.
Here, we use the convention $ x_{0}^{(\alpha)} := u $.
Note that $ \{ 1, \ldots, \len_\alpha - 1 \} = \emptyset $ if $ \ell_\alpha = 1 $.

In order to make the numerous generators feasible to handle for our singularity theoretic investigations, 
we deduce a new presentation with fewer generators (Prop.~\ref{Prop:NewPresentation}).
As in \cite{BFMS23,BFMS25}, continuant polynomials are a useful tool for this task. 
For a detailed reference on the latter, we refer to \cite{Muir}.

\begin{defi}
	\label{Def:Continuants}
	The {\em continuant polynomials} can be defined via  
	$ P_0 := 1 $, $ P_1 (z_1) := z_1 $ and 
	$ P_n (z_1, \ldots, z_n) := z_{n} P_{n-1} (z_1, \ldots, z_{n-1}) - P_{n-2} (z_1, \ldots, z_{n-2}) $, for $ n \geq 2 $.  
\end{defi}

This is an ad-hoc definition. 
In fact, continuant polynomials are defined as determinants of tri-diagonal matrices and we only require the special case where the non-zero non-diagonal entries are $ - 1 $.
For more details, we refer to \cite[Section~4.1]{BFMS23} or \cite[Section~3]{BFMS25}.

\begin{prop}
	\label{Prop:NewPresentation}
	We have: 
	\[
	\cA_\star^\prin 
	\cong 
	\KK_\c[\z, u , w ]
	/
	\langle 
	F_1, \ldots, F_m, G
	\rangle
	\ , 
	\]
	where $ \z := ( z_1^{(\alpha)}, \ldots, z_{\len_\alpha + 1}^{(\alpha)} \mid \alpha \in \{ 1, \ldots, m \}) $
	and
	\[
	\begin{array}{rcll}
		F_\alpha 
		& := & 
		P_{\len_\alpha +1} (z_1^{(\alpha)}, \ldots, z_{\len_\alpha + 1}^{(\alpha)}) 
		- \lambda_\alpha u
		\ , 
		& \mbox{ for } \alpha \in \{ 1, \ldots, m \} \ , 
		\\[10pt]
		G 
		& := &
		\displaystyle 
		u w - 
		\prod_{\alpha=1}^m P_{\len_\alpha} (z_1^{(\alpha)},  \ldots, z_{\len_\alpha}^{(\alpha)}) - \mu 
		\ ,
	\end{array} 
	\]
	with
	$ \lambda_\alpha := {\rho_{1}^{(\alpha)}}^{-1} $
	and
	$ \displaystyle  \mu := c_0^{-1} \cdot \Big( \prod_{\alpha=1}^m 
	\rho_{2}^{(\alpha)} \Big)^{-1} $,
	where we define $ \rho_i^{(\alpha)} := \rho_i^{(\alpha)}(\c^{(\alpha)}) $ for $ i \in \{ 1, \ldots, \alpha+1 \} $ via
	\[ 
		\rho_{\alpha+1}^{(\alpha)} := 1 
		,
		\quad  
		\rho_{\alpha}^{(\alpha)} := c_{\len_\alpha}^{(\alpha)} 
		,
		\quad
		\mbox{ and } 
		\quad 
		\rho_i^{(\alpha)} := c_{i}^{(\alpha)}
		\, \rho_{i+2}^{(\alpha)}
		\ \ 
		\mbox{ for } i < \alpha 
		\ . 
	\] 
\end{prop}

\begin{proof}
	Using that all $ c_i^{(\alpha)} $ are invertible, 
	we introduce the new variables
	\[
	\begin{array}{lcl} 
		\tx_i^{(\alpha)}
		:= 
		x_i^{(\alpha)} {\rho_{i+1}^{(\alpha)}}^{-1}
		& \
		\mbox{ and } 
		\ & 
		\ty_i^{(\alpha)}
		:= 
		y_i^{(\alpha)} \rho_{i+1}^{(\alpha)} \, {\rho_{i}^{(\alpha)}}^{-1}
		\ ,
	\end{array}
	\]
	for $ \alpha \in \{ 1, \ldots, m \} $ and $ i \in \{ 1, \ldots, \len_\alpha  \} $. 
	For $ \ell_\alpha = 1$, this leads to
	\[
		f_{\ell_\alpha}^{(\alpha)}
		\ = \
		f_1^{(\alpha)} 
		\ = \ 
		\rho_{1}^{(\alpha)} \cdot 
		\Big( 
		P_{2} (\tx_{1}^{(\alpha)}, \ty_{1}^{(\alpha)}) 
		- 
		\lambda_\alpha u  
		\Big) \ . 		
	\]
	For $ \ell_\alpha > 1 $, the coordinate transformation
	provides
	for the equations~\eqref{eq:gen-original} defining $ \cA_\star^\prin $:
	\[
	\begin{array}{rcll}
		f_{\len_\alpha}^{(\alpha)} 
		& := &
		\rho_\alpha^{(\alpha)} \cdot 
		\Big( 
		\tx_{\len_\alpha}^{(\alpha)} \ty_{\len_\alpha}^{(\alpha)} - \tx_{\len_\alpha-1}^{(\alpha)} - 1 
		\Big)   
		\ , 
		\\[10pt]
		f_{i}^{(\alpha)} 
		& = &
		\rho_{i}^{(\alpha)} \cdot 
		\Big(
		\tx_i^{(\alpha)} \ty_i^{(\alpha)} - \tx_{i-1}^{(\alpha)} - \tx_{i+1}^{(\alpha)} 
		\Big) 
		 \ ,
		&
		\ \ \ 
		\mbox{for } i \in \{ 2, \ldots, \len_\alpha - 1 \} \ .
	\end{array}
	\]
	We can drop these generators from \eqref{eq:gen-original}
	by applying the substitution
	\[
		\tx_i^{(\alpha)} 
		= 
		P_{\len_\alpha - i + 1 } (\tx_{\len_\alpha}^{(\alpha)}, \ty_{\len_\alpha}^{(\alpha)}, \ldots, \ty_{i + 1}^{(\alpha)}) 
		\ \ \
		\mbox{ for } 
		\alpha \in \{ 1, \ldots, m \}
		\mbox{ and }
		i \in \{ 1, \ldots, \len_\alpha-1\} 
		\ .
	\]
	Thus, we are left with $ g $ and 
	$ f_1^{(\alpha)} $ with $ \alpha \in \{ 1, \ldots, m \} $. 
	We have 
	\[
	\begin{array}{rcl}
		f_1^{(\alpha)} 
		& = & 
		\rho_{1}^{(\alpha)}
		\, \tx_{1}^{(\alpha)} \ty_{1}^{(\alpha)} - u - c_1^{(\alpha)} \rho_3^{(\alpha)}   \, \tx_{2}^{(\alpha)}  
		=
		\\[10pt]
		& = &
		\rho_{1}^{(\alpha)} \cdot 
		\Big( 
		P_{\len_\alpha+1} (\tx_{\len_\alpha}^{(\alpha)}, \ty_{\len_\alpha}^{(\alpha)}, \ldots, \ty_{1}^{(\alpha)}) 
		- 
		\lambda_\alpha u  
		\Big) \ . 
	\end{array} 
	\]
	Therefore, if we set
	$ ( z_1^{(\alpha)}, \ldots, z_{\len_\alpha + 1}^{(\alpha)} ) := (\tx_{\len_\alpha}^{(\alpha)}, \ty_{\len_\alpha}^{(\alpha)}, \ldots, \ty_{1}^{(\alpha)}) $,
	then we may replace
	$ f_1^{(\alpha)} $ by $ F_\alpha $ of the statement for every $ \alpha $.  
	Moreover, we get
	\[
		g = u v - c_0 \cdot \Big( \prod_{\alpha=1}^m 
		\rho_{2}^{(\alpha)} \Big) 
		\prod_{\alpha=1}^m P_{\len_\alpha} ( z_1^{(\alpha)}, \ldots, z_{\len_\alpha}^{(\alpha)}) - 1 
		\ . 
	\]
	By factoring $ \displaystyle c_0 \cdot \Big( \prod_{\alpha=1}^m 
	\rho_{2}^{(\alpha)} \Big)  $ from $ g $ 
	and defining $ \displaystyle w := v c_0^{-1} \cdot \Big( \prod_{\alpha=1}^m 
	\rho_{2}^{(\alpha)} \Big)^{-1}  = v \mu  $,
	we obtain $ G $.
	Hence, the assertion follows.  
\end{proof}

\begin{Bem}
	\label{Rk:eq_trivial1}
	\begin{enumerate}[(1)]
		\item 
	Observe that it is possible to eliminate the variable $ u $ via a suitable substitution,
	e.g., using the relation $ \lambda_1 u = P_{\len_1+1} ( z_1^{(1)}, \ldots, z_{\len_1 + 1}^{(1)} ) $.
	For $ \alpha \geq 2 $, we can then replace $ F_\alpha $ by 
	\[
		\lambda_1  P_{\len_\alpha+1}( z_1^{(\alpha)}, \ldots, z_{\len_\alpha + 1}^{(\alpha)} )
		- \lambda_\alpha P_{\len_1+1} ( z_1^{(1)}, \ldots, z_{\len_1 + 1}^{(1)} )
	\]
	and $ G $ becomes 
	\[
		P_{\len_1+1} ( z_1^{(1)}, \ldots, z_{\len_1 + 1}^{(1)} )
		\widetilde w - \prod_{\alpha=1}^m P_{\len_\alpha} ( z_1^{(\alpha)}, \ldots, z_{\len_\alpha}^{(\alpha)}) - \mu 
		\ ,
	\]
	where $ \widetilde w := w \lambda_1^{-1} $. 
	
	\item \label{it:trivial1}
	In the case of trivial coefficients, we have $ c_0 = c_i^{(\alpha)} = 1 $ 
	and thus $ \rho_i^{(\alpha)} = 1 $ as well as
	$ \lambda_\alpha = \mu = 1 $ 
	for all  $ \alpha \in \{ 1, \ldots, m \} $ and $ i \in \{ 1, \ldots, \len_\alpha  \} $. 
	\end{enumerate}
\end{Bem}

\section{Classification of the Singularities}
\label{Sec:Classi}

As in \cite[Section~1.4]{BFMS25},
we handle the coefficients $ \c $ by considering $ \Spec(\cA_\star^\prin) $ as family above $ S := \Spec(\KK_\c) $
(recall that $ \KK_\c = \KK [\c^\pm] $).
Hence, for a closed point $ \eta \in S $, 
we classify the singularities of the fiber 
$ \phi^{-1} (\eta ) $ of $ \phi \colon \Spec(\cA_\star^\prin) \to S $ in this section.
For convenience in notation, we fix a closed point $ \eta \in S $ and we set
\[  
	X := \phi^{-1} (\eta) 
	\subseteq \AA_\kappa^{\len_1+\cdots + \len_m+m+2} 
	\ ,
\] 
where $ \kappa = \kappa (\eta ) $ denotes the residue field of the local ring $ \mathcal{O}_{S,\eta} $.
For explicit defining equations for $ X $ and its embedding into $ \AA_\kappa^{\len_1+\cdots + \len_m+m+2} $, 
we use the presentation of $ \cA_\star^\prin $ deduced in Prop.~\ref{Prop:NewPresentation}. 
We abuse notation and use the same symbols $ F_\alpha, G, \lambda_\alpha $ and so on (defined in Prop.~\ref{Prop:NewPresentation}) for the respective elements in the fiber. 

In order to describe the singular locus of $ X $, we need some notation.
 
 \begin{defi}
 	\label{Def:partition}
 	For the positive integers $ \len_1, \ldots, \len_m \in \ZZ_+ $ and the invertible elements $\lambda_1, \ldots, \lambda_m \in \kappa^\times $,
 	we define: 
 	
 	\begin{enumerate}[(1)]
 	\item \label{item:compatible}
 	For $ \alpha, \beta \in \{ 1, \ldots, m \} $ 
 	we say that 
 	$ (\len_\alpha, \lambda_\alpha) $ and $ (\len_\beta, \lambda_\beta) $ 
 	are {\em compatible}
 	if 
 	$ \alpha = \beta $ or if the following conditions hold:
 	\begin{enumerate}[(a)]
 		\item
 		$ \len_\alpha = 2 \halflen_\alpha - 1 $ and $\len_\beta = 2 \halflen_\beta - 1  $ are both odd, 
 		for $ \halflen_\alpha, \halflen_\beta \in \ZZ_+ $
 		and
 		
 		\item
 		$ \lambda_\alpha (-1)^{\halflen_\alpha} = \lambda_\beta (-1)^{\halflen_\beta}  $. 
 	\end{enumerate}
 
	\item 
	For $ \alpha \in \{ 1, \ldots, m \} $,
	we define 
	\[
		I_\alpha := 
		\{
			\beta \in \{ 1, \ldots, m \}
			\mid 
			 (\len_\alpha, \lambda_\alpha) 
			\mbox{ and }
			(\len_\beta, \lambda_\beta) 
			\mbox{ are compatible} 
			\,
		\}
	\]
	and $ \sigma(I_\alpha) := \inf\{ \beta \in I_\alpha \}  $.
	Note that $ I_{\sigma(I_\alpha)} = I_\alpha $ if $ I_\alpha \neq \emptyset $.
 	
 	\item\label{Def:partition-item}
 	Define $ \alpha_1, \ldots, \alpha_s \in \{1, \ldots, m \} $ 
 	via:
 		for all $ \alpha \in \{ 1, \ldots, m \} $ with $ |I_{\alpha}| \geq 2 $ there is a unique $ \alpha_i $ with 
 		$ \alpha_i = \sigma(I_\alpha) $.
 	Moreover, we set $ I_0 := \{ 1, \ldots, m \} \setminus \bigcup_{i=1}^s I_{\alpha_i} $. 
 	Hence, this defines a partition of the index set into disjoint subsets:
 	\[
	 	\{ 1 , \ldots , m \}  = I_0 \cup I_{\alpha_1} \cup \cdots \cup I_{\alpha_s}
	 	\ .
 	\]
 \end{enumerate}
 \end{defi}
 
 \begin{example}
 	\label{Ex:Comaptible}
	\label{it:compat_A_1}
 	Suppose that $ \ell_\alpha = 2 k_\alpha - 1 $ 
 	and $ \lambda_\alpha = (-1)^{k_\alpha} $
 	for all $ \alpha \in \{ 1 , \ldots, m \} $. 
 	Then, all $ (\ell_\alpha, \lambda_\alpha) $ are compatible and thus $ \{ 1 , \ldots, m \} = I_1 $. 
 \end{example}

Recall that we introduced $ \displaystyle  \mu := c_0^{-1} \cdot \Big( \prod_{\alpha=1}^m 
	\rho_{2}^{(\alpha)} \Big)^{-1} $ in Prop.~\ref{Prop:NewPresentation}.

\begin{Thm}[Singular locus]
	\label{Thm:Sing} 
	The singular locus of $ \XX $ is described as follows:
	\begin{enumerate}[(1)]
		\item\label{Prop:Sing_origin} 
		If $ \len_\alpha = 2 \halflen_\alpha $ is even for all $ \alpha \in \{ 1 , \ldots , m \} $
		and $ (-1)^{\halflen_1 + \cdots + \halflen_m} + \mu  = 0 $,
		then 
		\[
		\Sing ( \XX )  \cong D \ , 
		\] 
		where $ D $ is the origin of $ \AA_\kappa^{\len_1 + \cdots + \len_m+m+2} $.
		
		\item\label{Prop:Sing_components} 
		Using the notation of Definition~\ref{Def:partition}\eqref{Def:partition-item}, if $ s \geq 1 $, 
		then 
		\[  
			\Sing ( \XX ) \cong  \bigcup_{i=1}^s  C_i 
		\	,
		\qquad
		\mbox{ with } 
		\quad 
		C_i := C_{\alpha_i} := \bigcup_{\substack{\alpha,\beta \in I_{\alpha_i} \\ \alpha \neq \beta }} C_{\alpha,\beta}
		\ ,
		\]
		where 
		\[
		C_{\alpha,\beta}
		:=
		V (
		\z^{(\alpha)}, 
		\z^{(\beta)}, 
		u - \xi_u, 
		w - \xi_w 	) 	 
		\ \cap \
		\bigcap_{\substack{\gamma = 1 \\ \gamma \neq \alpha,\beta}}^m 
		V ( P_{\len_\gamma +1} ( z_1^{(\gamma)}, \ldots, z_{\len_\gamma + 1}^{(\gamma)} ) 
		- \xi_\gamma  )
		\ ,
		\]
		for 
		$ \z^{(\alpha)} := (z_1^{(\alpha)}, \ldots, z_{\len_\alpha +1}^{(\alpha)} ) $, $  \z^{(\beta)} := (z_1^{(\beta)}, \ldots, z_{\len_\beta +1}^{(\beta)} ) $, 
		$ \xi_u := (-1)^{\halflen_\alpha}  \lambda_\alpha^{-1} $,
		$ \xi_w := (-1)^{\halflen_\alpha}\lambda_\alpha\mu $,
		and $ \xi_\gamma := (-1)^{\halflen_\alpha} \lambda_\gamma \lambda_\alpha ^{-1}  $.

		\item 
		In all other cases, $ \Sing(X) = \emptyset $, so $ X $ is regular.
	\end{enumerate}
\end{Thm}

Note that $ C_{\alpha,\beta} = C_{\beta,\alpha} $ since $(\ell_\alpha, \lambda_\alpha) $ and $(\ell_\beta, \lambda_\beta) $ are compatible. 
	Further,
observe that $ C_i \cap C_j = \emptyset $ for $ i,j \in \{ 1, \ldots, s \} $ with $ i \neq j $ in Thm.~\ref{Thm:Sing}\eqref{Prop:Sing_components}:
This follows since $ \alpha \in I_{\alpha_i} $ and $ \widetilde \alpha \in I_{\alpha_j} $ with $ i \neq j $ are not compatible, 
i.e.,
$ (-1)^{\halflen_\alpha}  \lambda_\alpha^{-1} \neq (-1)^{\halflen_{\widetilde \alpha}}  \lambda_{\widetilde \alpha}^{-1} $.

\begin{example}
	Let us pick up Example~\ref{Ex:Comaptible}. 
		If  $ \ell_\alpha = 2 k_\alpha - 1 $ 
		and $ \lambda_\alpha = (-1)^{k_\alpha} $
		for all $ \alpha \in \{ 1 , \ldots, m \} $,
		then 
		$ \Sing(X) = C_1 $ and 
		for $ \alpha, \beta \in I_1 = \{ 1, \ldots, m \} $ with $ \alpha \neq \beta $, 
		we obtain 
		\[
		C_{\alpha,\beta}
		:=
		V (
		\z^{(\alpha)}, 
		\z^{(\beta)}, 
		u - 1, 
		w - \mu 	) 	 
		\ \cap \
		\bigcap_{\substack{\gamma = 1 \\ \gamma \neq \alpha,\beta}}^m 
		V ( P_{\len_\gamma +1} ( z_1^{(\gamma)}, \ldots, z_{\len_\gamma + 1}^{(\gamma)} ) 
		- (-1)^{k_\gamma} )
		\ .
		\]
\end{example}

As preparation for the proof of Thm.~\ref{Thm:Sing}, we recall some facts about continuant polynomials
(Def.~\ref{Def:Continuants}) and their singularities.

\begin{lemma}
	\phantomsection 
	\label{Lem:ResultsContinuants}
	\begin{enumerate}[(1)]
		\item\label{Lem:ResultsContinuants_Pn(0)} 
		For $ k \in \ZZ_{\geq 0} $,
		we have
		$ P_{2k+1} (0, \ldots, 0 ) = 0 $,
		$ P_{2k+1} (0, \ldots, 0, z_n ) = (-1)^k z_n $,
		and $ P_{2k} (0, \ldots, 0) = P_{2k} (0, \ldots, 0, z_n) = (-1)^k $.
		
		\item\label{Lem:ResultsContinuants_deriv}
		\color{black} 
		For $ j \in \{ 1, \ldots, n \} $, we have
		\[
			\frac{\partial P_n(z_1, \ldots, z_n)}{\partial z_j} 
			= P_{j-1}(z_1, \ldots, z_{j-1})\cdot P_{n-j}(z_{j+1}, \ldots, z_n) 
			\ .
		\]
		
		\item\label{Lem:ResultsContinuants_consecutive} 
		There is no point at which two consecutive continuant polynomials vanish, i.e.,
		\[  
			V ( P_n (z_1, \ldots, z_n), P_{n-1} (z_1, \ldots, z_{n-1}) ) = \emptyset 
			\ .
		\]

		\item\label{Lem:ResultsContinuants_singu} 
		The variety $ \mathfrak X := V(P_n (z_1, \ldots, z_n ) + \lambda) \subseteq \AA_\kappa^n $ with $ \lambda \in \kappa $ is singular if and only if
		$ n = 2 \halflen $ is even and $ \lambda = (-1)^{\halflen+1} $.
		In the singular case,  $ \Sing ( \mathfrak{X} ) = V (z_1, \ldots, z_n )  $ and $  \mathfrak{X} $ has an isolated singularity of type $ A_1 $ at the origin.   
	\end{enumerate}
\end{lemma}

\begin{proof}
	First, we prove \eqref{Lem:ResultsContinuants_Pn(0)}:
	By \cite[Number~545]{Muir},
	all terms appearing in $ P_n(z_1, \ldots, z_n ) $ 
	are obtained from $ z_1 \cdots z_n $ by replacing every pair of
	consecutive $ z_i $ by $ -1 $. 
	Therefore the terms of order $ \leq 2 $ of $ P_n(z_1, \ldots, z_n ) $ are the following
	(cf.~\cite[Lemma~3.6]{BFMS25}):
	\begin{enumerate}[(i)]
		\item $ z_1 + z_3 + \cdots + z_{4m+1} $, \quad  if $ n = 4m+1 $,
		\item $ -1 + z_1 z_2  + z_1 z_4 + \cdots + z_1 z_{4m+2} + z_3 z_4 + \cdots + z_{4m+1} z_{4m+2} $, \quad  if $ n = 4m+2 $, 
		\item $ - z_1 - z_3 + \cdots - z_{4m+3} $, \quad  if $ n = 4m+3 $,
		\item $ 1 - z_1 z_2  - z_1 z_4 - \cdots - z_1 z_{4m+4} - z_3 z_4 - \cdots - z_{4m+3} z_{4m+4} $, \quad  if $ n = 4m+4 $. 
	\end{enumerate}	
	This provides Part~\eqref{Lem:ResultsContinuants_Pn(0)}.
	\\
	The formula for the derivative, Part \eqref{Lem:ResultsContinuants_deriv}, is a well known fact proven in \cite[Number 561 (4)]{Muir}.
	\\
	Part~\eqref{Lem:ResultsContinuants_consecutive} is a straight forward induction using the recursion of Def.~\ref{Def:Continuants},
	$ P_n (z_1, \ldots, z_n) = z_{n} P_{n-1} (z_1, \ldots, z_{n-1}) - P_{n-2} (z_1, \ldots, z_{n-2}) $.
	\\
	Part~\eqref{Lem:ResultsContinuants_singu} is proven in \cite[Lemma~3.7 and Prop.~3.8]{BFMS25}
\end{proof}

Next, we show the following generalization of Lemma~\ref{Lem:ResultsContinuants}\eqref{Lem:ResultsContinuants_singu}, 
which will be needed for the proof of Thm.~\ref{Thm:Sing}. 

\begin{prop}
	\label{Prop:ProdContSing}
	Let $ \mathfrak{X} := V(h) \subseteq \AA_\kappa^{\len_1+\cdots + \len_m} $,
	where
	\[
		h := \prod_{\alpha=1}^m P_{\len_\alpha} (z_1^{(\alpha)}, \ldots, z_{\len_\alpha}^{(\alpha)}) + \delta 
		\in \kappa [ z_1^{(\alpha)}, \ldots, z_{\len_\alpha}^{(\alpha)} \mid \alpha \in \{ 1, \ldots, m \}]
	\]
	with 
	$ \delta \in \kappa $. 
	If $ \delta $ is invertible, 
	then $ \mathfrak{X} $ is singular if and only if the following two conditions hold:
	\begin{enumerate}[(a)]
		\item[(a)] 
		$ \len_\alpha = 2 \halflen_\alpha $ is even, 
		with $  \halflen_\alpha \geq 1 $,
		for all $ \alpha \in \{ 1, \ldots, m \} $, and 
		
		\item[(b)] 
		$ (-1)^{\halflen_1 + \cdots + \halflen_m} + \delta  = 0 $.
	\end{enumerate}
	In the singular case, $ \Sing (\mathfrak{X}) $ is the origin and the hypersurface $ \mathfrak{X} $ has an isolated singularity of type $ A_1 $ at the origin. 
	\\
	On the other hand, if $ \delta = 0 $, then 
	$ \mathfrak{X} $ is a simple normal crossing divisor with $ m $ disjoint irreducible components, 
	i.e., after a suitable choice of coordinates $ h $ is of the form $ x_1 \cdots x_m $ with $ (x_1, \ldots, x_m) $ part of the chosen coordinates. 
\end{prop}

\begin{proof}
	First, suppose $ \delta $ is invertible. 
	If there exists an $ \beta \in \{ 1, \ldots, m \} $ such that $ \len_\beta $ is odd, 
	then the corresponding $ V ( P_{\len_\beta} (z_1^{(\beta)}, \ldots, z_{\len_\beta}^{(\beta)}) ) $ is regular by Lemma~\ref{Lem:ResultsContinuants}\eqref{Lem:ResultsContinuants_singu}.  
	Hence, we may introduce $ x_\beta :=  P_{\len_\beta} (z_1^{(\beta)}, \ldots, z_{\len_\beta}^{(\beta)}) $ as new variable.
	Clearly, $ h $ and $ \dfrac{\partial h}{\partial x_\beta} $ cannot vanish both at the same time since $ \delta $ is invertible. 
	In other words, $ V (h) $ is regular in this case. 
	\\
	Assume that $ \len_\alpha = 2 \halflen_\alpha $ are even for all $ \alpha \in \{ 1, \ldots, m \} $. 
	Notice that the vanishing of the partial derivatives of $ h $
	with respect to $ z_i^{(\alpha)} $ is equivalent to 
	\[
	\frac{\partial P_{\len_\alpha} (z_1^{(\alpha)}, \ldots, z_{\len_\alpha}^{(\alpha)})}{\partial z_i^{(\alpha)}} = 0
	\ , 
	\ \ \
	\mbox{ for all }
	\alpha \in \{1, \ldots, m \}
	\mbox{ and } i \in \{ 1 , \ldots, \len_\alpha \}
	\ . 
	\] 
	Analogous to the proof of \cite[Lemma~3.7]{BFMS23}, it follows that for a singular point, one has to have 
	$ z_i^{(\alpha)} = 0 $ for all $ i \in \{ 1, \ldots, \len_\alpha \} $ and all $ \alpha \in \{ 1, \ldots, m \} $.  
	Using this in $ h = 0 $, we get that additionally, $ (-1)^{\halflen_1 + \cdots + \halflen_m} + \delta  = 0 $ has to hold. 
	Hence, we achieved the statement on the singular locus if $ \delta $ is invertible. 
	(Since $ h \in \ZZ[z_1^{(\alpha)}, \ldots, z_1^{(\alpha)} \mid \alpha \in \{1, \ldots, m \}] $
	if $ \delta = (-1)^{\halflen_1 + \cdots + \halflen_m+1} $,
	derivatives with respect to a $ p $-basis of $ \kappa $ have no impact when determining the singular locus over non-perfect fields using Zariski's regularity criterion \cite[Theorem~11, p.~39]{Zar47}.)
	\\
	Let us classify the singularity in the singular case. 
	We introduce $ Q_{\len_\alpha} (z_1^{(\alpha)}, \ldots, z_{\len_\alpha}^{(\alpha)}) := P_{\len_\alpha} ( z_1^{(\alpha)}, \ldots, z_{\len_\alpha}^{(\alpha)} )  - (-1)^{\halflen_\alpha} $.
	Then, we get
	\[
	\begin{array}{c} 
		\displaystyle 
		h = 
		\prod_{\alpha=1}^m \Big( Q_{\len_\alpha} (z_1^{(\alpha)}, \ldots, z_{\len_\alpha}^{(\alpha)}) + (-1)^{\halflen_\alpha} \Big) + \delta
		=
		\sum_{\alpha=1}^m \epsilon_\alpha  Q_{\len_\alpha} (z_1^{(\alpha)}, \ldots, z_{\len_\alpha}^{(\alpha)})
		\ , 
		\\[15pt]
		\displaystyle 
		\mbox{where } \ 
		\epsilon_\alpha := 
		\prod_{\gamma=1}^{\alpha-1} (-1)^{\halflen_\gamma}
				\prod_{\beta = \alpha+1}^m \Big( Q_{\len_\beta} (z_1^{(\beta)}, \ldots, z_{\len_\beta}^{(\beta)}) + (-1)^{\halflen_\beta} \Big)
		\ . 
	\end{array}
	\]
	Note that, locally at the origin, all $ \epsilon_\alpha $ are units.
	Since each $ Q_{\len_\alpha} (z_1^{(\alpha)}, \ldots, z_{\len_\alpha}^{(\alpha)}) $ corresponds to a hypersurface singularity of type $ A_1 $ in $ \AA_\kappa^{\len_\alpha} $ (Lemma~\ref{Lem:ResultsContinuants}\eqref{Lem:ResultsContinuants_singu}),
	it follows that $ \mathfrak{X} $ has a singularity of type $ A_1 $ at the origin.  
	  
	 Finally, if $ \delta = 0 $,
	 Lemma~\ref{Lem:ResultsContinuants}\eqref{Lem:ResultsContinuants_singu} provides that each $ V (P_{\len_\alpha} ( z_1^{(\alpha)}, \ldots, z_{\len_\alpha}^{(\alpha)} ) ) $ is regular. 
	 This implies the missing part of the assertion. 
\end{proof}

\begin{proof}[\bf \em Proof of Thm.~\ref{Thm:Sing}]
	Recall that the polynomials $ F_\alpha, G \in \kappa[\z, u, w ] $ of the presentation deduced in Prop.~\ref{Prop:NewPresentation} are  
	\[
		\begin{array}{ll}
			F_\alpha  = P_{\len_\alpha +1} ( z_1^{(\alpha)}, \ldots, z_{\len_\alpha + 1}^{(\alpha)} ) 
			- \lambda_\alpha u
			\ ,
			&
			\mbox{ for } \alpha \in \{ 1, \ldots, m \} 
			\ ,
			\\[3pt]
			
			 G = \displaystyle 
			u w - 
			\prod_{\alpha=1}^m P_{\len_\alpha} (z_1^{(\alpha)}, \ldots, z_{\len_\alpha}^{(\alpha)}) - \mu
			\ .  
		\end{array} 
	\]
	First, observe that the singular locus relative to $ \kappa $,
	which we denote by $ \Sing_\kappa (X) $, 
	is determined by the maximal minor of the Jacobian matrix 
	\[ 
		J : =\Jac(F_1, \ldots, F_m, G; \z , u, w )
		\ .
	\] 	
	Analogous to the proof of Prop.~\ref{Prop:ProdContSing} we will see that derivatives with respect to constants do not contribute to the singular locus
	when determining the singular locus over non-perfect fields using Zariski's regularity criterion \cite[Theorem~11, p.~39]{Zar47}. 
	\\
	Recall that the singularity is given by $ m + 1 $ equations so maximal minors are determined by choosing $ m + 1 $ variables.
	Using Lemma~\ref{Lem:ResultsContinuants}\eqref{Lem:ResultsContinuants_deriv},
	the maximal minors corresponding to the derivatives by $ (z_{\len_\alpha+1}^{(\alpha)}, u, w  \mid \alpha \in \{1, \ldots, m \} \setminus \{\beta\}) $, for some fixed $ \beta $,
	and with respect to $ (z_{\len_\alpha+1}^{(\alpha)}, u  \mid \alpha \in \{1, \ldots, m \}) $
	provide the equations
	\begin{equation}
		\label{eq:sing_1}
		u \cdot \frac{\displaystyle \prod_{i=1}^m P_{\len_\alpha} (z_1^{(\alpha)}, \ldots, z_{\len_\alpha}^{(\alpha)})}{P_{\len_\beta} (z_1^{(\beta)}, \ldots, z_{\len_\beta}^{(\beta)})}
		=
		w \prod_{i=1}^m P_{\len_\alpha} (z_1^{(\alpha)}, \ldots, z_{\len_\alpha}^{(\alpha)})
		= 0 
	\end{equation}
	for the singular locus, where $ \beta \in \{ 1, \ldots, m \} $.
	
	First, assume $ u = w = 0 $. 
	Then the vanishing of $ F_\alpha $ and $ G $ is equivalent to  
	\[
		P_{\len_\alpha +1} ( z_1^{(\alpha)}, \ldots, z_{\len_\alpha + 1}^{(\alpha)} )  
		= \displaystyle 
		\prod_{\alpha=1}^m P_{\len_\alpha} ( z_1^{(\alpha)}, \ldots, z_{\len_\alpha}^{(\alpha)} )  + \mu  = 0 
		\ . 
	\]
	In particular, $ \dfrac{\partial F_\alpha}{\partial z_{\len_\alpha+1}^{(\alpha)}} = P_{\len_\alpha} ( z_1^{(\alpha)}, \ldots, z_{\len_\alpha}^{(\alpha)} )  $ cannot vanish for $ \alpha \in \{ 1, \ldots, m \} $
	by Lemma~\ref{Lem:ResultsContinuants}\eqref{Lem:ResultsContinuants_consecutive}.
	Hence, the maximal minors with respect to 
	$ (z_{\len_1+1}^{(1)},z_{\len_2+1}^{(2)}, \ldots, z_{\len_m+1}^{(m)}, *) $,
	where $* $ is one of the remaining variables, but not $ u $ or $ w $,
	lead to conditions
	$ \dfrac{\partial G}{\partial *} = 0 $ 
	for singularities to arise.
	These equations can be rephrases as 
	\[ 
		\frac{\partial}{\partial *} (\displaystyle 
		\prod_{\alpha=1}^m P_{\len_\alpha} ( z_1^{(\alpha)}, \ldots, z_{\len_\alpha}^{(\alpha)} )  + \mu)
		= 0 
		\ . 
	\] 
	We deduced that 
	all entries of the row of Jacobian matrix $ J $ corresponding to the derivatives of $ G $ 
	vanish.
	This means that all remaining minors are automatically zero. 
	Therefore 
	\[
		D := 
		V(u,w, P_{\len_\alpha +1} ( z_1^{(\alpha)}, \ldots, z_{\len_\alpha + 1}^{(\alpha)} ) 
		\mid \alpha \in \{ 1, \ldots, m \})
		\cap 
		\Sing (\prod_{\alpha=1}^m P_{\len_\alpha} ( z_1^{(\alpha)}, \ldots, z_{\len_\alpha}^{(\alpha)} ) + \mu)
	\]
	is contained in the singular locus relative to $ \kappa $.
	Prop.~\ref{Prop:ProdContSing} provides that $ D $ is non-empty if and only if 
	$ \len_\alpha = 2 \halflen_\alpha $ is even, for all $ \alpha \in \{ 1, \ldots, m \} $,
	and $ (-1)^{\halflen_1 + \cdots + \halflen_m} + \mu  = 0 $.
	If  $ D \neq \varnothing $, then $ D = V(\z, u,w) $ is the origin by
	Prop.~\ref{Prop:ProdContSing} and  Lemma~\ref{Lem:ResultsContinuants}\eqref{Lem:ResultsContinuants_Pn(0)}.
	Observe that derivatives with respect to a $ p $-basis of $ \kappa $  have no impact on the component $ D $ so that $ D $ is a component of $ \Sing (\XX) $. 
		
	Next, \eqref{eq:sing_1} provides the equations  
	$ u = P_{\len_\alpha} (z_1^{(\alpha)}, \ldots, z_{\len_\alpha}^{(\alpha)}) = 0 $  
	for a possible singularity. 
	Since $ u = F_\alpha = 0 $,
	it is required to have
	$ P_{\len_\alpha +1} ( z_1^{(\alpha)}, \ldots, z_{\len_\alpha + 1}^{(\alpha)} ) = 0 $
	for a singularity,
	which is impossible 
	by Lemma~\ref{Lem:ResultsContinuants}\eqref{Lem:ResultsContinuants_consecutive}.
	Hence, no singularities are detected in this case.

	The last case emerging from \eqref{eq:sing_1} is that 
	\[
		P_{\len_\alpha} (z_1^{(\alpha)}, \ldots, z_{\len_\alpha}^{(\alpha)})
		=
		P_{\len_\beta} (z_1^{(\beta)},  \ldots, z_{\len_\beta}^{(\beta)})
		=
		0
		\ ,
	\]
	for two fixed $ \alpha, \beta \in \{ 1, \ldots, m \} $ with $ \alpha \neq \beta $. 
 	Using this in $ G = 0 $, we obtain that $ uw - \mu = 0 $.
	Further, all entries in the row of $ J $ corresponding to $ G $ are zero except for the derivatives by $ u $ and $ w $. 
	Notice that the column corresponding to the derivatives with respect to $ u $ consists only of invertible elements.
	Furthermore,
	the column for the derivatives by $ w $
	is of the form $ (0, \ldots, 0 , u)^T $.
	This implies that it is sufficient to consider the maximal minors arising from
	derivatives by $ (u,w,*^{(\sigma)} \mid \sigma \in \{ 1, \ldots , m \} \setminus \{ \tau \}) $, for $ \tau \in \{ 1, \ldots, m \} $
	and where $ *^{(\sigma)} $ denotes any of the $ z $-variables with superscript $ ( \sigma ) $.
	This provides the additional following conditions for a singularity:
	\begin{equation}
		\label{eq:pf_uw*}
		\prod_{\substack{\sigma=1\\\sigma \neq \tau}}^m
		\frac{\partial P_{\len_\sigma +1} ( z_1^{(\sigma)}, \ldots, z_{\len_\sigma + 1}^{(\sigma)} ) }{\partial z_{i(\sigma)}^{(\sigma)}}
		= 0
		\ 
		,
		\hspace{0.5cm} 
		\begin{array}{ll}
			\mbox{for} &
			\tau \in \{ 1, \ldots, m \} 
			\mbox{ and}
			\\
			&
			i(\sigma) \in \{ 1, \ldots \len_{\sigma} + 1 \} \ . 
		\end{array}
	\end{equation}
	From this, we obtain the components $ C_{\alpha,\beta, \sigma, \tau} \subseteq \Sing_\kappa (X) $ of the form
	\[
		C_{\alpha,\beta, \sigma, \tau} := 
		\Sing_\kappa (F_\sigma) 
		\cap 
		\Sing_\kappa (F_\tau) 
		\cap V ( u w  - \mu
		,
		P_{\len_\alpha} (z_1^{(\alpha)}, \ldots, z_{\len_\alpha}^{(\alpha)})
		,
		P_{\len_\beta} (z_1^{(\beta)},  \ldots, z_{\len_\beta}^{(\beta)}))
		\cap 
		\XX \ .
	\] 
	Since $ C_{\alpha,\beta, \sigma, \tau} \subset C_{\sigma, \tau, \sigma, \tau} $, 
	it suffices to consider the component $ C_{\alpha,\beta}  := C_{\alpha,\beta, \alpha,\beta} $ for the chosen $ \alpha, \beta $. 
	Using 
	$ 
	\lambda_\beta F_\alpha - \lambda_\alpha F_\beta
	=
	\lambda_\beta P_{\len_\alpha +1} ( z_1^{(\alpha)}, \ldots, z_{\len_\alpha + 1}^{(\alpha)} ) - \lambda_\alpha P_{\len_\beta +1} ( z_1^{(\beta)}, \ldots, z_{\len_\beta + 1}^{(\beta)} )
	$,
	we see that 
	\[
	C_{\alpha,\beta} = 
	\Sing_\kappa (\lambda_\beta F_\alpha - \lambda_\alpha F_\beta )
	\cap V ( u w  - \mu)
	\cap 
	\XX \ .
	\] 
	Taking into account Lemma~\ref{Lem:ResultsContinuants}\eqref{Lem:ResultsContinuants_singu}, 
	it follows that 
	$ V (\lambda_\beta F_\alpha - \lambda_\alpha F_\beta) \subseteq \AA_\kappa^{\len_\alpha + \len_\beta + 2} $ 
	is singular if and only if $ \alpha, \beta $ are compatible
	(i.e., 
	$ \len_\alpha + 1 = 2 \halflen_\alpha $, $ \len_\beta + 1 = 2 \halflen_\beta $,
	and
	$ \lambda_\beta (-1)^{\halflen_\alpha} - \lambda_\alpha (-1)^{\halflen_\beta} = 0 $,
	see Def.~\ref{Def:partition}\eqref{item:compatible}). 
	In the singular case, $ V (\lambda_\beta F_\alpha - \lambda_\alpha F_\beta) \subseteq \AA_\kappa^{\len_\alpha + \len_\beta + 2} $ has an isolated singularity of type $ A_1 $ at the origin $ V (\z^{(\alpha)}, \z^{(\beta)}) $.
	\\
	On $ V (\z^{(\alpha)}, \z^{(\beta)}) $, the vanishing
	$ F_\alpha = 0 $ is equivalent to $ (-1)^{\halflen_\alpha}  - \lambda_\alpha u = 0 $.
	In conclusion, we get if $ C_{\alpha, \beta} $ is non-empty 
	then
	$ 
		C_{\alpha, \beta} = 
		V (
		\z^{(\alpha)}, 
		\z^{(\beta)}, 
		u - \xi_u, 
		w - \xi_w )
		\cap 
		\XX  	
	$, 
	for $ \xi_u := (-1)^{\halflen_\alpha}  \lambda_\alpha^{-1} $,
	$ \xi_w := (-1)^{\halflen_\alpha}\lambda_\alpha\mu $.
Notice that the part $ \cap \XX $ in $ C_{\alpha,\beta} $ could be replaced by 
	$ \cap 
	\bigcap_{\gamma \neq \alpha,\beta}
	V ( P_{\len_\gamma +1} ( z_1^{(\gamma)}, \ldots, z_{\len_\gamma + 1}^{(\gamma)} ) 
	- \xi_\gamma  ) $,
	for
	$ \xi_\gamma := (-1)^{\halflen_\alpha} \lambda_\gamma \lambda_\alpha^{-1}  $.

	For $ \kappa $ perfect, this completes the proof of Thm.~\ref{Thm:Sing}.
	Suppose $ \kappa $ is not perfect.
	We want to argue that $ C_{\alpha,\beta} \subseteq \Sing(X) $ if it is non-empty.  
	Let $ \mathcal{B}_\kappa $ be any $ p $-basis of $ \kappa $.
	On $ C_{\alpha,\beta} $, 
	the two rows corresponding to $ F_\alpha $ and $ F_\beta $ of the Jacobian matrix with respect to $ (\z, u, w, \mathcal{B}_\kappa) $
	are linear dependent (over $ \ZZ $) 
	because all non-zero entries are determined by derivatives of $ \lambda_\alpha u $ respectively $ \lambda_\beta u = (-1)^{\kappa_\beta - \kappa_\alpha} \lambda_\alpha u $.
	Therefore, all maximal minors vanish on $ C_{\alpha,\beta} $
	and so, $ C_{\alpha,\beta} $ is a component of the singular locus of $ X $, as claimed. 
\end{proof}

For the classification of the singularities, we require the following result.

\begin{lemma}
	\label{Lem:intersec_of_An}
	Let $ n \in \ZZ_{\geq 2} $ be a positive integer with $ n \geq 2 $. 
	For $ a \in \{ 1, \ldots, n \} $, let
	\[
		g_a (\t^{(a)}) := \sum_{i=1}^{q_a} t_{2i-1}^{(a)} t_{2i}^{(a)} \in \kappa [\t^{(a)}] = \kappa [t_1^{(a)}, \ldots, t_{2q_a}^{(a)}]
	\]
	be the polynomial defining an $ A_1 $-hypersurface singularity in $ \Spec (\kappa [\t^{(a)}]) \cong \AA_\kappa^{2q_a} $. 
	Let 
	\[
		\mathcal{X} :=  \mathcal{X}^{(n)} := V ( g_a (\t^{(a)}) - g_b (\t^{(b)}) \mid a, b \in \{ 1, \ldots, n\} )
		\subset \AA_\kappa^{2(q_1+ \cdots + q_n)} =: \mathcal{Z} 
		\ .
	\] 
	Define $ \mathcal{Y}_0 \subset  \mathcal{Y}_1 \subset  \cdots \subset \mathcal{Y}_{n-2} \subset \mathcal{X} $
	by
	\[
		\mathcal{Y}_j 
		\ := \
		\mathcal{Y}_j^{(n)} 
		\ := 
		\bigcup_{(a_1, \ldots, a_{n-j})}
		\Big( V(\t^{(a_1)}, \ldots, \t^{(a_{n-j})}) \cap \mathcal{X} \Big)  
		\ , 
	\]
	where the union ranges over all $ (n-j) $-tuples $ (a_1, \ldots, a_{n-j}) $
	with $ a_1 < a_2 < \cdots < a_{n-j} $. 
	\\
	The following sequence of blow-ups is an embedded resolution of singularities for $ \mathcal{X} $:
	\[
		\mathcal{Z}_{n-1} \stackrel{\pi_{n-2}}{\longrightarrow}
		\cdots 
		\stackrel{\pi_2}{\longrightarrow}
		\mathcal{Z}_2 \stackrel{\pi_1}{\longrightarrow}
		\mathcal{Z}_1 \stackrel{\pi_0}{\longrightarrow}
		\mathcal{Z}
		\ , 
	\]
	where $ \pi_0 $ is the blow-up with center $ \mathcal{Y}_0 $
	and $ \pi_j \colon \mathcal{Z}_{j+1} \to \mathcal{Z}_{j} $ is the blow-up with center the strict transform of $ \mathcal{Y}_j $ in $ \mathcal{Z}_{j} $, for $ j \in \{ 1, \ldots, n-2 \} $.
\end{lemma}

In the following, we call a singularity of the type described in Lemma~\ref{Lem:intersec_of_An} an 
{\em intersection of compatible $ A_1 $-hypersurface singularities centered at $ V (\t^{(a)} \mid a \in \{ 1, \ldots, n \}) $}.
The reader may wonder why the adjective ``compatible" is used, we address this in Example~\ref{Ex:not-compatible-A1}.

\begin{proof}[Proof of Lemma~\ref{Lem:intersec_of_An}]
	First observe that $ \mathcal{Y}_0 $ is the origin
	and for $ j \in \{ 1, \ldots, n-2 \} $,
	\[  
		\mathcal{Y}_j 
		= 
		\bigcup_{(a_1, \ldots, a_{n-j})}
		\Big( 
		V(\t^{(a_1)}, \ldots, \t^{(a_{n-j})}) \cap V(g_b(\t^{(b)}) \mid b \in \{1,\ldots, m \} \setminus \{a_1, \ldots, a_{n-j}\})
		\Big)  
		\ .
	\]  
	We prove the assertion by induction on $ n $.
	For $ n = 2 $, $ \mathcal{X} $ is an $ A_1 $-hypersurface singularity. 
	Hence, blowing up the origin $ \mathcal{Y}_0 $ resolves the singularities, as desired. 
	
	Assume $ n > 2 $.
	Perform the blow-up $ \pi_0 $ with center $ \mathcal{Y}_0^{(n)} $.
	Without loss of generality, we only consider the chart where the exceptional divisors is determined by $ t_1^{(n)} $ --- all other charts are analogous. 
	The coordinates transform as 
	$ t_1^{(n)} = {t_1^{(n)}}' $ and 
	$ t^{(a)}_j = {t_1^{(n)}}' {t^{(a)}_j}' $ for $ a \in \{ 1, \ldots, n \} $ and $ j \in \{ 1, \ldots, 2q_a \} $ with $ (a,j) \neq (n,1) $.   
	The exceptional divisor of the blow-up is locally given as $ V ({t_1^{(n)}}') $.
	We denote by $ (.)' $ the strict transform of $ (.) $ in the given chart.  
	Notice that $ (\mathcal{X}^{(n)})' $ is determined by the 
	strict transforms of the system of generators $ ( g_a (\t^{(a)}) - g_b (\t^{(b)}) \mid a, b ) $.
	Since $ g_n (\t^{(n)}) $ defines an $ A_1 $-hypersurface singularity, 
	it is resolved in the given charts. 
	On the other hand, for $ a \neq n $, we get
	$ (g_a (\t^{(a)}))'= g_a ({\t^{(a)}}')$. 
	Hence, if we choose $ a \in \{ 1, \ldots, n \} $, $ a \neq n $, then $ V( (g_n (\t^{(n)}))' - g_a ({\t^{(a)}}')) $ is regular
	and if we substitute  $ (g_n (\t^{(n)}))' = g_a ({\t^{(a)}}') $ in the remaining generators, 
	we obtain
	\[ 
		(\mathcal{X}^{(n)})' \cong V ( g_a ({\t^{(a)}}') - g_b ({\t^{(b)}}') \mid a, b \in \{ 1, \ldots, n-1\} ) \times_{\Spec(\kappa) } \AA_\kappa^{2q_n-1}
		\ . 
	\]
	In other words, 
	$ (\mathcal{X}^{(n)})' \cong \mathcal{X}^{(n-1)} \times_{\Spec(\kappa)} \AA_\kappa^{2q_n-1}  \subset \AA_\kappa^{2(q_1+ \cdots + q_n) - 1} $.
	Furthermore, along this isomorphism, we have 
	$ ( \mathcal{Y}_j^{(n)} )' \cong \mathcal{Y}_{j-1}^{(n-1)}\times_{\Spec(\kappa)} \AA_\kappa^{2q_n-1} $ for $ j \in \{ 2, \ldots, n -2 \} $. 
	Thus, we can apply induction and the statement follows. 
\end{proof}

\begin{example}
	\label{Ex:not-compatible-A1}
	Let $ f,g \in \kappa[t_1, t_2, t_3, t_4, t_5, t_6] $ be given by 
	\[ 
	f = t_1 t_2 + t_3 t_4 
	\qquad 
	\mbox{ and } 
	\qquad 
	g = t_5 t_6 + t_2 (t_3 + t_4^3 ) 
	\ .
	\] 
	Observe that $ V(f) $ is an $ A_1 $-hypersurface singularity.
	If we introduce $ s_3 := t_3 + t_4^2 $, then $ g = t_5 t_6 + t_2 s_3 $,
	so $ V (g) $ is also an $ A_1 $-hypersurface singularity.
	On the other hand, $ V (f,g) $ is not an intersection of compatible $ A_1 $-hypersurface since there are no coordinates such that 
	$ f $ and $ g $ are both of the form $ \sum_{i=1}^2 w_{2i-1} w_{2i} $.
	\\  
	This can also be seen by blowing up the origin $ V (t_1, \ldots, t_6 ) $.
	Consider the chart where $ t_i = t_i' t_4' $ for $ i \in \{ 1 , \ldots, 6 \} \setminus \{4\} $ and $ t_4 = t_4' $.
	The strict transform of $ f $ is $ f' = t_1' t_2' + t_3' $
	and the one of $ g $ is $ g' = t_5' t_6' + t_2' (t_3' + t_4'^2) $.
	Introduce the new variable $ s := t_3' + t_1' t_2' $ and substitute $ t_3' = s - t_1' t_2' $.
	We get $ f' = s $ 
	and $ g' = t_5' t_6' + t_2' (s -t_1't_2' + t_4'^2) $. 
	Hence, $ V (f',g') = V(s,\, t_5' t_6' + t_2' (t_4'^2 -t_1't_2')) $
	which is not isomorphic to an $ A_1 $-hypersurface singularity.  
\end{example}

In the next step, we classify the singularities of $ \XX $ 
which we described in Thm.~\ref{Thm:Sing}.
Clearly, a component 
$ C_{\alpha,\beta} $ may be singular itself 
since $ V ( P_{\len_\gamma +1} ( z_1^{(\gamma)}, \ldots, z_{\len_\gamma + 1}^{(\gamma)} ) 
- \xi_\gamma  ) $ can be singular, cf.~Lemma~\ref{Lem:ResultsContinuants}\eqref{Lem:ResultsContinuants_singu}.

\begin{Thm}[Classification]
	\label{Thm:ClassSing} 
	Using the notation of Thm.~\ref{Thm:Sing}, we have \
	\begin{enumerate}[(1)]
		\item 
		\label{it:origin}
		If $ \ell_\alpha = 2 k_\alpha $ for all $ \alpha \in \{ 1, \ldots, m \} $
		and $ (-1)^{k_1 + \cdots + k_m} + \mu = 0 $,
		then $ \XX $ is locally at the origin $ D $ isomorphic to a hypersurface singularity of type $ A_1 $. 
		
		\item 
		\label{it:c_i}
		If $ C_{\alpha_i} = \bigcup C_{\alpha,\beta} $ is contained in $ \Sing(X) $,
		then $ X $ is locally at $ C_{\alpha_i} $ isomorphic to a cylinder over an intersection of compatible $ A_1 $-hypersurface singularities centered at $ V ( \z^{(\beta)} \mid \beta \in I_{\alpha_i} ) $,
		where $ C_{\alpha_i} $ corresponds to $ \mathcal{Y}_{n-2} $ in the notion of Lemma~\ref{Lem:intersec_of_An} for $ n := |I_{\alpha_i}| $.
	\end{enumerate}
\end{Thm}

\begin{proof}
	Recall (cf.~Prop.~\ref{Prop:NewPresentation}) 
	that we work with the presentation of $ \XX $ in $ \AA_\kappa^{\len_1+\cdots+\len_m+ m + 2} $ given by 
	\[
	\begin{array}{ll}
		F_\alpha  = P_{\len_\alpha +1} ( z_1^{(\alpha)}, \ldots, z_{\len_\alpha + 1}^{(\alpha)} ) 
		- \lambda_\alpha u
		\ ,
		&
		\mbox{ for } \alpha \in \{ 1, \ldots, m \} 
		\ ,
		\\[3pt]
		
		G = \displaystyle 
		u w - 
		\prod_{\alpha=1}^m P_{\len_\alpha} (z_1^{(\alpha)}, \ldots, z_{\len_\alpha}^{(\alpha)}) - \mu
		\ .  
	\end{array} 
	\]
	 
	First, consider the origin $ D \in \AA_\kappa^{\len_1 + \cdots + \len_m + m + 2} $, 
	assuming that $ D \in \Sing (\XX) $. 
	Recall from the proof of Thm.~\ref{Thm:Sing} that
	\[
			D = 
		V(u,w, P_{\len_\alpha +1} ( z_1^{(\alpha)}, \ldots, z_{\len_\alpha + 1}^{(\alpha)} ) 
		\mid 1 \leq \alpha \leq m )
		\cap 
		\Sing (\prod_{\alpha=1}^m P_{\len_\alpha} ( z_1^{(\alpha)}, \ldots, z_{\len_\alpha}^{(\alpha)} ) + \mu)
	\]
	By Lemma~\ref{Lem:ResultsContinuants}\eqref{Lem:ResultsContinuants_consecutive}, $ P_{\len_\alpha +1} ( z_1^{(\alpha)}, \ldots, z_{\len_\alpha + 1}^{(\alpha)}) = 0 $ implies that 
	$ P_{\len_\alpha } ( z_1^{(\alpha)}, \ldots, z_{\len_\alpha}^{(\alpha)} ) $ is invertible.
	Hence, locally at $ D $, we may introduce the new variables 
	$ x_{\len_\alpha + 1}^{(\alpha)} :=  P_{\len_\alpha +1} ( z_1^{(\alpha)}, \ldots, z_{\len_\alpha + 1}^{(\alpha)} ) $ replacing $ z_{\len_\alpha + 1}^{(\alpha)} $,
	for all $ \alpha \in \{ 1, \ldots, m \} $. 
	On the other hand, 
	$ V ( \prod_{\alpha=1}^m P_{\len_\alpha} ( z_1^{(\alpha)}, \ldots, z_{\len_\alpha}^{(\alpha)} ) + \mu ) \subseteq  \AA_\kappa^{\len_1+\cdots + \len_m} $ has a singularity of type $ A_1 $ at the origin of $ \AA_\kappa^{\len_1+\cdots + \len_m} $ by Prop.~\ref{Prop:ProdContSing}.
	\\
	In conclusion, we may eliminate $ F_\alpha $ for all $ \alpha \in \{ 1, \ldots, m \} $ and the remaining $ G $ defines a hypersurface singularity of type $ A_1 $ in $ \AA_\kappa^{\len_1+\cdots+\len_m+2} $. 
	Hence, we have shown part \eqref{it:origin}.
	
	Next, let us look locally at $ C_i = C_{\alpha_i} = \bigcup_{\alpha,\beta \in I_{\alpha_i} \, : \,  \alpha \neq \beta } C_{\alpha, \beta} $. 
	Since $ C_i \cap C_j = \varnothing $ for $ i \neq j $, 
	there are no other components of $ \Sing(\XX) $ than $ C_i$ that are seen here.
	Recall from the proof of Thm.~\ref{Thm:Sing} that
	$ 
		C_{\alpha,\beta} = 
		\Sing (\lambda_\beta F_\alpha - \lambda_\alpha F_\beta )
		\cap V ( u w  - \mu)
		\cap 
		\XX 
	$. 
	Therefore, $ C_i \subset V (uw- \mu) $
	and locally at $ C_i $, 
	$ u $ and $ w $ are invertible. 
	This implies that we may introduce $ w' := u^{-1} G $ taking the role of $ w $.
	Thus we may neglect the equation $ G = 0 $ for the classification,
	while the other equations remain untouched. 	
	\\
	Set $ \alpha := \alpha_i $
	and substitute $ u = \lambda_\alpha^{-1} P_{\len_\alpha +1} ( \z^{(\alpha)}) $.
	Then $ X $ is locally at $ C_i $ given by 
	\[
		V (
		\lambda_\alpha P_{\len_\beta +1} ( \z^{(\beta)})  
		- \lambda_\beta P_{\len_\alpha +1} ( \z^{(\alpha)})
		\mid \beta \in \{ 1, \ldots, m \} \setminus \{ \alpha \}
		)	
		\ . 
	\]
	For $ \beta \notin I_\alpha $,
	we have either $ \len_\beta + 1 $ is odd, or 
	$ \len_\beta + 1  = 2 \halflen_\beta $ and $ \lambda_\beta (-1)^{\halflen_{\alpha}} - \lambda_{\alpha} (-1)^{\halflen_\beta} \neq 0 $.
	Hence  
	the hypersurface $ V (\lambda_\alpha P_{\len_\beta +1} ( \z^{(\beta)})  
	- \lambda_\beta P_{\len_\alpha +1} ( \z^{(\alpha)})) $ is regular
	and this element may be neglected for the classification of the singularity of $ \XX $ at $ C_i $. 	 
	\\
	On the other hand, if $ \beta \in I_\alpha $, 
	then the hypersurface 
	$ V (\lambda_\alpha P_{\len_\beta +1} ( \z^{(\beta)})  
	- \lambda_\beta P_{\len_\alpha +1} ( \z^{(\alpha)}) 
	) $ has a singularity of type $ A_1 $ at $ V (\z^{({\alpha})}, \z^{(\beta)}) $.
	By Lemma~\ref{Lem:ResultsContinuants}\eqref{Lem:ResultsContinuants_singu}, there exist local coordinates $ \t^{(\alpha)} $ 
	(substituting $ \z^{(\alpha)} $)
	such that 
	$
	P_{\len_\alpha +1} ( \z^{(\alpha)}) = \sum_{j=1}^{\halflen_\alpha} t_{2j-1}^{(\alpha)} t_{2j}^{(\alpha)} + (-1)^{\halflen_\alpha}
	$
	where $ \len_\alpha + 1 = 2 k_\alpha $.
	The analogous statement holds for $ P_{\len_\beta +1} ( \z^{(\beta)}) $. 
	In conclusion,
	we obtain that  
	$ V (
	\lambda_\alpha P_{\len_\beta +1} ( \z^{(\beta)})  
	- \lambda_\beta P_{\len_\alpha +1} ( \z^{(\alpha)})
	\mid \beta \in I_\alpha \setminus \{ \alpha \}
	) $
	defines an intersection of compatible $ A_1 $-hypersurface singularities
	centered $ V( \z^{(\beta)} \mid 
	\beta \in I_\alpha ) $ which provides assertion \eqref{it:c_i}.
\end{proof}

\section{Combinatorial Considerations}
\label{Sec:Combi}

In this section, we investigate the combinatorics of the singularities of a star-shaped cluster algebra. 
In \cite{BFMS23}, the number of irreducible components of the singular locus was determined for a star-shaped cluster algebra with $ m $ rays of length 1 and with trivial coefficients (Fig.~\ref{Fig:Ex_star_quiver_BFMS23}) as
\[  
	\binom{m}{2} 2^{m-2} = m(m-1) 2^{m-3} \ . 
\]
(Note that in \cite[Thm.~6.3]{BFMS23} the formula was given for $ m = n-1 $).

As we have seen in Thm.~\ref{Thm:ClassSing},
the appearing singularities are intersections of compatible $ A_1 $-hypersurfaces.
Therefore, we focus on combinatorics of this type of singularities.

\begin{Bem}
	The intersection of compatible $ A_1$-hypersurface singularities centered at $ V (\t^{(a)} \mid a \in \{ 1, \ldots, n \}) $ of Lemma~\ref{Lem:intersec_of_An} corresponds
	to a star-shaped cluster algebra 
	whose underlying quiver has $ n $ rays which are of lengths $ \len_a = 2 q_a - 1 $ 
	and for which the coefficients fulfill $ \lambda_\alpha = (-1)^{q_a} $
	for $ a \in \{ 1, \ldots, n \} $.
	In terms of Def.~\ref{Def:partition}\eqref{Def:partition-item} this means, $ m =n $, $ s = 1 $, $ |I_1| = 1 $, and $ I_0 = \emptyset $, see also Example~\ref{Ex:Comaptible}.
\end{Bem}

Recall the setting of Lemma~\ref{Lem:intersec_of_An}:
For $ n \in \ZZ_{\geq 2} $, 
\[
\begin{array}{l}  
	\displaystyle 
	\mathcal{X} \ = \  \mathcal{X}^{(n)} \ =  \ V ( g_a (\t^{(a)}) - g_b (\t^{(b)}) \mid a, b \in \{ 1, \ldots, n\} )
	\subset \AA_\kappa^{2(q_1+ \cdots + q_n)} = \mathcal{Z}
	\ ,  
	\\[5pt]
	\displaystyle 
	\mathcal{Y}_j 
	\ = \
	\mathcal{Y}_j^{(n)} 
	\ = 
	\bigcup_{(a_1, \ldots, a_{n-j})}
	\Big( V(\t^{(a_1)}, \ldots, \t^{(a_{n-j})}) \cap \mathcal{X} \Big)  
	\ , 
	\quad j \in \{ 0, \ldots, n - 2 \} 
	\ , 
\end{array} 
\]
where
$
g_a (\t^{(a)}) = \sum_{i=1}^{q_a} t_{2i-1}^{(a)} t_{2i}^{(a)} \in \kappa [\t^{(a)}] = \kappa [t_1^{(a)}, \ldots, t_{2q_a}^{(a)}] $
for $ a \in \{ 1, \ldots, n \} $
and the union ranges over all $ (n-j) $-tuples $ (a_1, \ldots, a_{n-j}) $
with $ a_1 < a_2 < \cdots < a_{n-j} $. 

The numerical data that we are interested in is the following:
 
\begin{defi}
	\label{Def:cA}
	Set $ \q := (q_1, \ldots, q_n) $. 
	For $ j \in \{ 0, \ldots, n - 2 \} $,
	we define
	\[  
		\mathcal{A} (n;\q; j) 
	\] 
	to be the number of irreducible components of $ \mathcal{Y}_j $
	and
	\[
		\mathcal{A} (n;\q) := \ \sum_{j=0}^{n-2} \mathcal{A} (n;\q; j) 
		\ . 
	\]
\end{defi}

Notice that $ \mathcal{Y}_{n-2} $ is the singular locus of $ \mathcal{X} $. 
Hence, the result of \cite{BFMS23} can be reformulated as $ \mathcal{A} (n;1, \ldots, 1; n-2) = n (n-1) 2^{n-3} $.

\begin{prop}
	\label{Prop:cA}
	Let the situation be as in Def.~\ref{Def:cA}.
	Set $ \alpha := \# \{ a \in \{ 1, \ldots, n \} \mid q_a = 1 \} $.
	For $ j \in \{ 0, \ldots, n - 2 \} $,
	we have 
	\[
		\mathcal{A} (n;\q; j) 
		\ = \ 
		\sum_{k=0}^{n-j} 
		\binom{n-\alpha}{n-j-k} \binom{\alpha}{k} 2^{\alpha - k}
	\]
	and
	\[
	\mathcal{A} (n;\q) 
	\ = \ 
	3^\alpha 2^{n-\alpha}  
	- ( n + 1) 2^{\alpha} 
	+ \alpha 2^{\alpha - 1}
	\ .
	\]
\end{prop}

\begin{proof}
	The description of $ \mathcal{Y}_j $ provides components of the form
	\[ 
		V(\t^{(a_1)}, \ldots, \t^{(a_{n-j})}) \cap \mathcal{X} 
		\ = \ 
		V(\t^{(a_1)}, \ldots, \t^{(a_{n-j})}, 
		g_b (\t^{(b)}) \mid 
		b \in \{ 1, \ldots, n\} \setminus \{ a_1, \ldots, a_{n-j} \}
		)
		\ . 
	\]
	These are not necessarily irreducible components. 
	If $ q_b = 1 $, 
	we have $ g_b (\t^{(b)}) = t_{1}^{(b)} t_{2}^{(b)} $. 
	Hence, for every such $ b $, we obtain two components.
	\\
	Fix $ k \in \{ 0, \ldots, n-j \} $.
	Assume that $ k $ many elements of $ a_1, \ldots, a_{n-j} $ fulfill $ q_a = 1 $.
	There are $ \binom{\alpha}{k} $ possibilities to choose these $ k $ elements,
	while we have $ \binom{n-\alpha}{n-j-k} $ choices for  the rest of $ a_1, \ldots, a_{n-j} $ 
	(for which necessarily $ q_a > 1 $ holds).
	This implies that the number of $ b \in \{ 1 , \ldots, n \}  \setminus \{ a_1, \ldots, a_{n-j} \} $
	with $ q_b = 1 $ is equal to $ \alpha - k $.
	Hence, $ 2^{\alpha-k} $ irreducible components are created through them.   
	In conclusion, for fixed $ k $, this consideration leads to the factor
	\[
		\binom{n-\alpha}{n-j-k} \binom{\alpha}{k} 2^{\alpha - k}
		\ . 
	\]
	By summing over all possible $ k $, we obtain the claimed formula for $ \mathcal{A} (n;\q; j)  $.

	Notice that $ \binom{\alpha}{k} = 0 $ if $ k > \alpha $ and $ \binom{n-\alpha}{n-j-k} = 0 $ for $ k > n- j $.
	Thus, the identities
	\[
	\mathcal{A} (n;\q; j) 
	\ = \ 
	\sum_{k=0}^{n} 
	\binom{n-\alpha}{n-j-k} \binom{\alpha}{k} 2^{\alpha - k}
	\ = \ 
	\sum_{k=0}^{\alpha} 
	\binom{n-\alpha}{n-j-k} \binom{\alpha}{k} 2^{\alpha - k}
	\]
	hold.
	
	By definition, $ \mathcal{A} (n;\q) $ is the sum over all $ \mathcal{A} (n;\q;j) $.
	By rearranging the terms in the sum, 
	we get
	\[
		\mathcal{A} (n;\q)
		\ = \ 
		\sum_{j=0}^{n-2}  
		\sum_{k=0}^{\alpha} 
		\binom{n-\alpha}{n-j-k} \binom{\alpha}{k} 2^{\alpha - k}
		\ = \ 
		\sum_{k=0}^{\alpha} 
		 \binom{\alpha}{k} 2^{\alpha - k}
		 \sum_{j=0}^{n-2}   \binom{n-\alpha}{n-j-k}
		 \ = \ 
	\]
	\[
		\ = \ 
		\sum_{k=0}^{\alpha} 
		\binom{\alpha}{k} 2^{\alpha - k}
		\sum_{e = 2-k}^{n-k}   \binom{n-\alpha}{e}
		\ \stackrel{k \leq \alpha}{=} \
		\sum_{k=0}^{\alpha} 
		\binom{\alpha}{k} 2^{\alpha - k}
		\sum_{e = 2-k}^{n-\alpha}   \binom{n-\alpha}{e} 
		\ . 
	\]
	We have 
	$ \displaystyle \sum_{e=0}^{n-\alpha}   \binom{n-\alpha}{e}  = 2^{n-\alpha} $
	and $ \displaystyle \sum_{k=0}^{\alpha} 
	\binom{\alpha}{k} 2^{- k} = (1 + 2^{-1})^\alpha = 3^\alpha 2^{-\alpha} $.
	Note that the sum in consideration starts with $ e = 2 - k $, so we have to subtract terms for $ k \in \{0,1\} $.
	Therefore, we deduce 
	\[
		\mathcal{A} (n;\q)
		\ = \ 
		\sum_{k=0}^{\alpha} 
		\binom{\alpha}{k} 2^{n - k} 
		- \alpha 2^{\alpha - 1} 
		- (n+1 - \alpha) 2^{\alpha}
		\ = \
		3^\alpha 2^{n-\alpha}  
		- ( n + 1) 2^{\alpha} 
		+ \alpha 2^{\alpha - 1}
		\ ,
	\]
	as claimed.
\end{proof}

Fig.~\ref{Fig:Table}
shows the values of 
$  \mathcal{A} (n;\q) 
= 
3^\alpha 2^{n-\alpha} - ( n + 1) 2^{\alpha} + \alpha 2^{\alpha - 1} $
for small $ (\alpha, n) $.
Recall that $ n \geq 2 $ and $ 0 \leq \alpha \leq n $.

 \begin{figure}
\[
\begin{array}{|c||c|c|c|c|c|c|c|c|c|c|c|}
	\hline 
	\textcolor{white}{\Big|}
	n \backslash \alpha
	& 0 & 1 & 2 & 3 & 4 & 5 & 6 & 7 & 8 & 9 
	\\
	\hline \hline 
	\textcolor{white}{\Big|}
	2
	& 1 & 1 & 1 & - & - & - & - & - & - & - 
	\\
	\hline 
	\textcolor{white}{\Big|}
	3
	& 4 & 5 & 6 & 7 &  - & - & - & - & - & - 
	\\
	\hline 
	\textcolor{white}{\Big|}
	4
	& 11 & 15 & 20 & 26 & 33 &  - & - & - & - & - 
	\\
	\hline 
	\textcolor{white}{\Big|}
	5
	& 26 & 37 & 52 & 72 & 98 &  131 & - & - & - & - 
	\\
	\hline 
	\textcolor{white}{\Big|}
	6
	& 57 & 83 & 120 & 172 & 244 & 342 & 473 & - & - & - 
	\\
	\hline 
	\textcolor{white}{\Big|}
	7
	& 120 & 177 & 260 & 380 & 552 & 796 & 1.138 & 1.611 & -& -
	\\
	\hline 
	\textcolor{white}{\Big|}
	8
	& 247 & 367 & 544 & 804 & 1.184 & 1.736 & 2.532 & 3.670 & 5.281 & - 
	\\
	\hline 
	\textcolor{white}{\Big|}
	9
	& 502 & 749 & 1.116 & 1.660 & 2.464 & 3.648 & 5.384 & 7.916 & 11.586 & 16.867 
	\\
	\hline 
	\end{array}
\]
\caption{Values of 
	$  \mathcal{A} (n;\q) 
	= 
	3^\alpha 2^{n-\alpha} - ( n + 1) 2^{\alpha} + \alpha 2^{\alpha - 1} $ 
	where the value $ n $ is fixed in each row
	and the value of $ \alpha $ is fixed in each column. \label{Fig:Table}} 	
\end{figure}

Let us have a look at boundary cases in terms of the value of $ \alpha \in \{ 0, \ldots, n \} $.

\begin{cor}
	\phantomsection 
	\label{cor:bondary}
	\begin{enumerate}[(1)]
		\item 
		For $ \alpha = 0 $,
		we get:
		\[  
		\mathcal{A} (n;\q; j) = \binom{n}{j}
		\quad 
		\mbox{ and }
		\quad   
		\mathcal{A} (n;\q) = 2^n - ( n + 1 ) \ .
		\] 
		
		\item 
		If $ \alpha = n $,
		then $ \q = (1,\ldots, 1) $ and
		\[  
		\mathcal{A} (n;1,\ldots, 1; j) = 		 \binom{n}{j} 2^{j}
		\quad 
		\mbox{ and }
		\quad   
		\mathcal{A} (n;1,\ldots, 1) = 
		3^n  - 2^{n}
		- n 2^{n - 1} 
		\ .
		\]  
	\end{enumerate}
\end{cor}

We have $ 2^n - ( n + 1 ) =  A(n,2) $, where 
	\[  
		A(n,k) = \sum_{i=0}^k (-1)^i \binom{n+1}{i} (k-i)^n 
	\] 
are the Eulerian numbers,
e.g., see \cite[Thm.~C, p.~243]{Comtet}
(see also \cite[A000295]{OEIS}).
\\
The integer sequence 
$ ( 3^n  - (n + 2) 2^{n - 1} )_{n \geq 0 } =: (b_0, b_1, b_2, \ldots ) $  
is the binomial transform of $ ( 2^n - ( n + 1 ) )_{n\geq 0}  =: (a_0, a_1, a_2, \ldots ) $ 
(see \cite[A066810]{OEIS}),
i.e.,
\[  
b_m = \sum_{k=0}^m  \binom{m}{k} a_k 
\ .
\]

Observe that we have implicitly shown in the proof of Prop.~\ref{Prop:cA} that the irreducible components of $ {\mathcal Y}_j $ have codimension
\[
\sum_{a \in \{ a_1, \ldots, a_{n-j} \} } 2 q_{a} + j 
\ .
\]
In particular, the codimension may vary for a fixed $ j $ depending on the values of $ q_a $. 

\begin{Qu}
	Let $ c \in \ZZ_+ $.
	How many irreducible component of codimension $ c $ do exist in $ {\mathcal Y}_j $, resp.~in $ \bigcup_{j=0}^{n-2} {\mathcal Y}_j $? 
	For which values of $ c $ are these counts non-zero? 
\end{Qu}

In the special case where $ q_a = q $ for some fixed $ q \in \ZZ_+ $ and all $ a \in \{ 1, \ldots, n \} $, 
all irreducible components of $ {\mathcal Y}_j $ are of codimension 
\[  
2q(n-j) + j = 2 qn - (2q-1) j  \ . 
\]
Therefore, Prop.~\ref{Prop:cA} provides an answer in this particular situation. 
Furthermore, since we have $ j \in \{ 0, \ldots, n- 2\} $, there is also a clear constraint on the codimensions that may appear.

\section{Beyond Singularities of Star Cluster Algebras}
\label{Sec:beyond}

A natural questions that comes up is how the singularity theory and corresponding combinatorial aspects do look like for cluster algebras with more general underlying quivers. 
In order to motivate possible further research, we outline some examples illustrating first directions going beyond star cluster algebras.
To keep things simple, we will not specify the orientation of edges in the figures, but we point out that it may become important to take this detail into account when investigating the corresponding commutative algebras. 

\begin{example}
	\label{Ex.1}
	For star-shaped quivers the fact that there is only one central vertex connecting finitely many quivers of type $ A_\ell $ for varying $ \ell $, 
	allowed to reduce the number of equations defining the corresponding variety. 
	Fig.~\ref{Fig:beyond1} shows a first example, where this is not possible since there are two vertices (marked with a black filling)
	of valency $ 3 $. 
	
	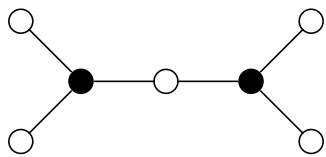
\begin{figure}[h!] 
		\scalebox{0.75}{
			\begin{tikzpicture}[-,>=stealth',shorten >=1pt,auto,node distance=1.5cm, thick,main node/.style={circle,draw,white}]

				\node[main node] (0) { };
				\node[main node] (1) [below left of=0] { };
				\node[main node] (2) [above left of=0] { };
				\node[main node] (3) [right of = 0] { };
				\node[main node] (00) [right of = 3] { };
				\node[main node] (01) [above right of=00] { };
				\node[main node] (02) [below right of=00] { };
				
				\draw[fill=black] (0) circle (6pt);
				\draw (1) circle (6pt);
				\draw (2) circle (6pt);
				\draw (3) circle (6pt);
				\draw[fill=black] (00) circle (6pt);
				\draw (01) circle (6pt);
				\draw (02) circle (6pt);
				
				\path
				(1) edge (0)
				(2) edge (0)
				(3) edge (0)
				(3) edge (00)
				(01) edge (00)
				(02) edge (00);

			\end{tikzpicture}		
			%
			%
			%
			%
			%
			%
			%
	}
	\caption{Example of the underlying graph of a quiver with two vertices of valency $ 3 $.}
	\label{Fig:beyond1}
	\end{figure}
	
	In particular, new phenomena for the possibility to reduce the number of exchange relations may appear and have to be understood.
	A particular example of this phenomenon are quivers whose underlying graphs are phylogenetic trees \cite{phylo}.  
\end{example}

\begin{example}
	\label{Ex.2}
	Since the underlying graphs of star cluster algebras are trees, the corresponding quivers are in particular acyclic. 
	Hence, another possible direction is quivers 
	whose underlying graphs contain cycles.
	The latter could imply that the quiver contains oriented cycles in which case additional exchange relations have to be taken into account. 
	First examples are illustrated in Fig~\ref{Fig:beyond2}. 
	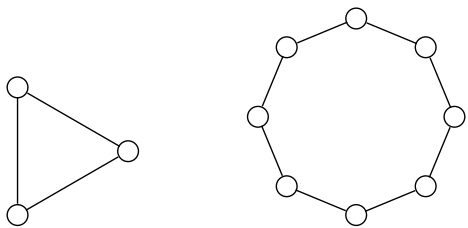
\begin{figure}[h!] 
		\scalebox{0.65}{
			\begin{tikzpicture}[-,>=stealth',shorten >=1pt,auto,node distance=1.5cm, thick,main node/.style={circle,draw,white}]

				\node[main node] (1) at (0:1.5cm)  { };
				\node[main node] (2) at (120:1.5cm) { };
				\node[main node] (3) at (240:1.5cm)  { };

				\draw (1) circle (6pt);
				\draw (2) circle (6pt);
				\draw (3) circle (6pt);
				
				\path
				(1) edge (2)
				(2) edge (3)
				(3) edge (1);

			\end{tikzpicture}
			\hspace{2cm}
			\begin{tikzpicture}[-,>=stealth',shorten >=1pt,auto,node distance=1.5cm, thick,main node/.style={circle,draw,white}]

				\node[main node] (1) at (0:2cm)  { };
				\node[main node] (2) at (45:2cm) { };
				\node[main node] (3) at (90:2cm)  { };
				\node[main node] (4) at (135:2cm)  { };
				\node[main node] (5) at (180:2cm)  { };
				\node[main node] (6) at (225:2cm)  { };
				\node[main node] (7) at (270:2cm)  { };
				\node[main node] (8) at (315:2cm)  { };

				\draw (1) circle (6pt);
				\draw (2) circle (6pt);
				\draw (3) circle (6pt);
				\draw (4) circle (6pt);
				\draw (5) circle (6pt);
				\draw (6) circle (6pt);
				\draw (7) circle (6pt);
				\draw (8) circle (6pt);
				
				\path
				(1) edge (2)
				(2) edge (3)
				(3) edge (4)
				(4) edge (5)
				(5) edge (6)
				(6) edge (7)
				(7) edge (8)
				(8) edge (1);

			\end{tikzpicture}
		}
		\caption{First examples of underlying graphs of quivers that include cycles.}
		\label{Fig:beyond2}
	\end{figure}
\end{example}

\begin{example}
	\label{Ex.3}
	Once there is an understanding of the singularities arising from quivers as described in Examples~\ref{Ex.1} and~\ref{Ex.2}
	they can be combined to get examples of quivers with more involved underlying graphs. 
	In Fig.~\ref{Fig:beyond3}, we provide an example.
	
	\begin{figure}[h!] 
		\scalebox{0.8}{
			\begin{tikzpicture}[-,>=stealth',shorten >=1pt,auto,node distance=1.5cm, thick,main node/.style={circle,draw,white}]
				
				\node[main node] (01) { };
				\node[main node] (1) [below left of=01] { };
				\node[main node] (2) [above left of=01] { };
				\node[main node] (3) [above of=01] { };
				\node[main node] (4) [above right of = 01] { };
				\node[main node] (5) [right of = 4] { };
				\node[main node] (02) [below right of = 5] { };
				\node[main node] (03) [below right of = 02] { };
				\node[main node] (6) [above right of = 02] { };
				\node[main node] (7) [below right of = 03] { };
				\node[main node] (8) [right of = 03] { };
				\node[main node] (9) [right of = 8] { };
				\node[main node] (10) [below left of = 03] { };
				\node[main node] (11) [below left of = 02] { };
				
				\draw (01) circle (6pt);
				\draw (02) circle (6pt);
				\draw (03) circle (6pt);
				\draw (1) circle (6pt);
				\draw (2) circle (6pt);
				\draw (3) circle (6pt);
				\draw (4) circle (6pt);
				\draw (5) circle (6pt);
				\draw (6) circle (6pt);
				\draw (7) circle (6pt);
				\draw (8) circle (6pt);
				\draw (9) circle (6pt);
				\draw (10) circle (6pt);
				\draw (11) circle (6pt);
				
				\path
				(1) edge (01)
				(2) edge (01)
				(3) edge (01)
				(4) edge (01)
				(5) edge (4)
				(5) edge (02)
				(6) edge (02)
				(03) edge (02)
				(7) edge (03)
				(8) edge (03)
				(9) edge (8)
				(10) edge (03)
				(02) edge (01)
				(02) edge (11);

			\end{tikzpicture}
		}
		\caption{Example of an underlying graph of a quiver with cycles and multiple vertices of valency $ \geq 3 $.}
		\label{Fig:beyond3}
	\end{figure}
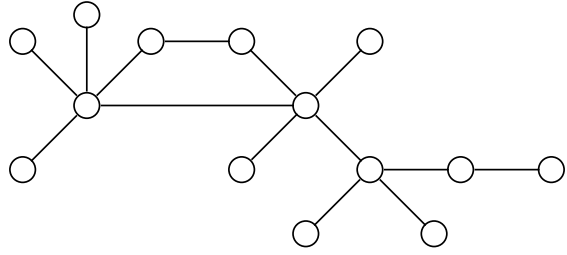 
\end{example}

\begin{example}
	\label{Ex.4}
	In the present article, we only considered quivers whose underlying graph was simply-laced,
	i.e., there was at most one edge between two vertices.
	Hence, it is natural to ask how things change if multiple edges are admitted? 
	This already appeared in the classification of the singularities of finite cluster type in \cite{BFMS23,BFMS25}.  
	\\
	Star cluster algebras allowed to provide a unified picture for the singularities corresponding to finite cluster type $ A, D, E $.
	Is it possible to generalize this presentation of the singularities  to include all finite cluster types, i.e., also those with multiple edges between two vertices? 
	\\
	Of course, with the possibility of multiple edges there are numerous examples by hand which can become arbitrarily complicated. 
	In particular, in combination with vertices of high valency and cycles within the underlying graph. 
	In Fig.~\ref{Fig:beyond4} we show two examples.
		
	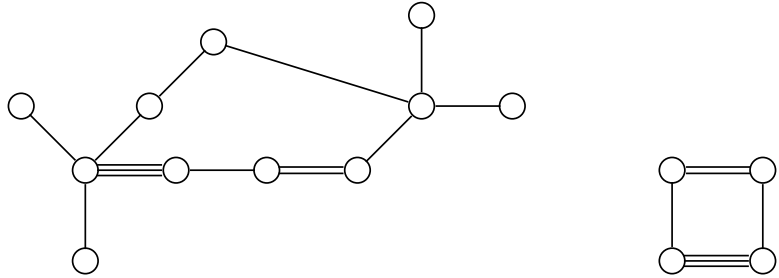
\begin{figure}[h!] 
		\scalebox{0.8}{
			\begin{tikzpicture}[-,>=stealth',shorten >=1pt,auto,node distance=1.5cm, thick,main node/.style={circle,draw,white}]

				\node[main node] (0) { };
				\node[main node] (2) [below of=0] { };
				\node[main node] (3) [above left of=0] { };
				\node[main node] (4) [above right of = 0] { };
				\node[main node] (5) [above right of = 4] { };
				\node[main node] (6) [right of = 0] { };
				\node[main node] (7) [right of = 6] { };
				\node[main node] (8) [right of = 7] { };	
				\node[main node] (00) [above right of = 8] { };	
				\node[main node] (01) [right of = 00] { };	
				\node[main node] (02) [above of = 00] { };

				\draw (0) circle (6pt);
				\draw (00) circle (6pt);
				\draw (2) circle (6pt);
				\draw (3) circle (6pt);
				\draw (4) circle (6pt);
				\draw (5) circle (6pt);
				\draw (6) circle (6pt);
				\draw (7) circle (6pt);
				\draw (8) circle (6pt);
				\draw (01) circle (6pt);
				\draw (02) circle (6pt);
				
				\path
				(2) edge (0)
				(3) edge (0)
				(4) edge (0)
				(5) edge (4)
				(7) edge (6)
				(0) edge (6)
				(8) edge (00)
				(5) edge (00)
				(01) edge (00)
				(02) edge (00);
				
				\draw[transform canvas={yshift=-1.5pt}] (7) -- (8);
				\draw[transform canvas={yshift=1.5pt}] (7) -- (8);
				
				\draw[transform canvas={yshift=-2.5pt}] (0) -- (6);
				\draw[transform canvas={yshift=2.5pt}] (0) -- (6);

			\end{tikzpicture}
			\hspace{2cm}
			\begin{tikzpicture}[-,>=stealth',shorten >=1pt,auto,node distance=1.5cm, thick,main node/.style={circle,draw,white}]

				\node[main node] (0) { };
				\node[main node] (1) [below of=0] { };
				\node[main node] (2) [right of=1] { };
				\node[main node] (3) [above of = 2] { };
				
				\draw (2) circle (6pt);
				\draw (3) circle (6pt);
				\draw (0) circle (6pt);
				\draw (1) circle (6pt);
				
				\path
				(0) edge (1)
				(1) edge (2)
				(2) edge (3);
				
				\draw[transform canvas={yshift=-1.5pt}] (3) -- (0);
				\draw[transform canvas={yshift=1.5pt}] (3) -- (0);
				
				\draw[transform canvas={yshift=-2.5pt}] (1) -- (2);
				\draw[transform canvas={yshift=2.5pt}] (1) -- (2);

			\end{tikzpicture}
		}
		\caption{Examples of underlying graphs with multiple edges between two vertices.}
		\label{Fig:beyond4}
	\end{figure}
	
\end{example}

Finally, getting back to star cluster algebras: 
We have studied star cluster algebras from the viewpoint of singularity theory and deduced combinatorial aspects.
\\
Cluster algebras appear in numerous different areas of mathematics, e.g. in representation theory, mathematical physics, or symplectic geometry. 
Hence, we are pointed to:

\begin{Qu}
	Does the structure of star cluster algebras lead to new insights pointing towards more general results which are currently still hidden in other areas of mathematics? 
\end{Qu}


\newcommand{\etalchar}[1]{$^{#1}$}
\def\cprime{$'$}

\end{document}